\documentclass{article}
\usepackage{amssymb}
\usepackage{amsmath}
\usepackage{amsthm}
\usepackage{dsfont}
\usepackage{color}
\usepackage{hyperref}
\usepackage{fullpage} 
\usepackage{xcolor}
\usepackage{empheq}
\usepackage{caption}
\usepackage{subcaption} 
\usepackage{graphics} 
\usepackage{epsfig} 
\usepackage{graphicx}
\usepackage{epstopdf}
\usepackage{mathabx}
\usepackage{capt-of} 
\usepackage{multirow}
\usepackage{hhline}

\usepackage{hyperref} 
\usepackage[capitalise]{cleveref}
\usepackage{appendix}
\usepackage{color}
\definecolor{strcolor}{rgb}{0.6, 0.2, 0.6}
\definecolor{commentcolor}{rgb}{0.3125, 0.5, 0.3125}
\definecolor{keycol}{rgb}{0, 0, 1}

\usepackage{bbm}

\def \to{\rightarrow}

\def \N {\mathbb{N}}

\def \lv{\left\vert}
\def \rv{\right\vert}
\def \lc{\left(}
\def \rc{\right)}

\def \abs{|}

\usepackage{listings}
\def \to{\rightarrow}

\def \N {\mathbb{N}}

\def \lv{\left\vert}
\def \rv{\right\vert}
\def \lc{\left(}
\def \rc{\right)}

\def \abs{|}

\usepackage{enumitem}

\usepackage[linesnumbered,ruled,lined,noend]{algorithm2e}

\SetKwComment{Comment}{/* }{ */}
\SetKwRepeat{Do}{do}{while}

 \newtheorem{theorem}{\textbf{Theorem}}[section]
 \newtheorem{remark}[theorem]{\textbf{Remark}}
 \newtheorem{lemma}[theorem]{\textbf{Lemma}}
 \newtheorem{corollary}[theorem]{\textbf{Corollary}}
 \newtheorem{definition}[theorem]{\textbf{Definition}}
 
 \newtheorem{assumption}[theorem]{\textbf{Assumption}}

\makeatletter
\newcommand{\pushright}[1]{\ifmeasuring@#1\else\omit\hfill$\displaystyle#1$\fi\ignorespaces}
\newcommand{\pushleft}[1]{\ifmeasuring@#1\else\omit$\displaystyle#1$\hfill\fi\ignorespaces}
\makeatother

\gdef\AQ#1{}
\gdef\CQ#1{}
\title{Learning equilibria in Cournot mean field games of controls}
\date{}
\author{Fabio Camilli\thanks{University of Rome La Sapienza}, Mathieu Lauri\`ere\thanks{NYU Shanghai Frontiers Science Center of Artificial Intelligence and Deep Learning; NYU-ECNU Institute of Mathematical Sciences at NYU Shanghai}, Qing Tang\thanks{China University of Geosciences (Wuhan), tangqingthomas@gmail.com}}

\begin{document}
\maketitle

\begin{abstract}
We consider Cournot mean field games of controls, a model originally developed for the production of an exhaustible resource by a continuum of producers. We prove  uniqueness of the solution under general assumptions on the price function. Then, we prove convergence of a learning algorithm  which gives existence of a solution to the mean field games  system. The learning algorithm  is implemented with a suitable finite difference discretization to get a numerical method to the solution.  We supplement our theoretical analysis with several numerical examples.
\end{abstract}

\section{Introduction}
Mean field games (MFG) theory has been introduced in \cite{MR2352434,MR2295621} (see \cite{MR3967062,MR3752669} for a comprehensive introduction)  as a   tool  to study games with infinitely many agents and it is particularly suitable to model economic problems for market competition. While in the first formulation of the theory  the influence of the mass of agents on the single player occurs only through the  spatial distribution, a subsequent development, known as mean field games of controls,   considers the case in which the agent also reacts to decisions, i.e.   optimal controls, of the other agents (see \cite{MR4565014,MR3805247, MR3160525,MR4387201}). In this latter framework, a very interesting problem,   first proposed in \cite{gueant2010mean,MR2762362} and then widely studied in the literature,  is the one known as Cournot mean field games of controls,    a model for the production of an exhaustible resource by a continuum of producers.\par
To introduce the Cournot model, we start describing the control problem solved by the representative agent.  The state variable  $x \in  \mathbb{L}=(0,L)$  represents the inventory level. The dynamics of the agent, a producer, is given by the reflected stochastic differential equation
$$
	dX_\tau= -q_\tau\, d\tau + \sqrt{2\sigma^2(x)}dB_\tau-d\xi_{X_\tau},\,\,X_t= x,
$$
where $d\xi_{X_\tau}$ is the local time at $x=L$ and  $q_\tau=q(\tau,X_\tau)$ is the control variable which gives the instantaneous rate of production. For $t\in (0, T)$, letting
$ \tau^*:=\inf \{\tau\geq t:\,\,X_\tau\leq 0\}$
the value function is
\begin{equation}\label{valuefunction}
u(t,x)=\sup_{q_\tau \geq 0} {\mathbb E} \left\{\int_t^{T\wedge \tau^* } e^{-\lambda(\tau-t)}\Big(q_\tau \mathrm{P}(\tau)-\gamma q_\tau-\kappa q^2_\tau\Big) d\tau +e^{-\lambda({T\wedge \tau^* }-t)}u(X_{T\wedge \tau^*})\right\}.
\end{equation}
where $\lambda>0$ is the discount factor, $\mathrm{P}(\tau)$ is the price,  $\gamma q+\kappa q^2$,  with $\gamma,\,\kappa>0$,  the production cost. Given the  function $m(t,\cdot)$ which describes the distribution of the agents at time $t$, assume that
\begin{itemize}
	\item In equilibrium, all agents are using the same policy, which we denote by ${q^*}(\tau,\cdot)$, but they are heterogeneous in their inventories $X_\tau$;
	\item Each agent is infinitesimally small, i.e. the producers influence the market price $\mathrm{P}(\tau)$ only through the average production $\int_0^L {q^*}(\tau,x)m(\tau,x)dx$;
	\item Each agent is a price taker, taking as given the price $\mathrm{P}(\tau)$ from its ``belief'', hence the $\sup$ in \cref{valuefunction} is only over $q_\tau$.
\end{itemize}
Then, the Nash equilibria can be characterized by the MFG of controls system (cfr. \cite[Proposition 3.1, p. 26]{gueant2010mean}):
\begin{equation}\label{MFG}
\left\{\begin{aligned}
	(\textit{i}) \qquad  & \partial_t u +\sigma^2(x) \partial^2_{xx}u -\lambda u +\sup_{q\geq 0}\left\{-q\partial_xu+q\mathrm{P}(t)-\gamma q-\kappa q^2\right\}=0\,\, &&{\rm in}\,\,Q, \\
	            & \mathrm{P}(t)=P\Big( t,\int_0^L {q^*}(t,x)m(t,x)dx\Big), u(t,0)=0,\,\partial_xu(t,L)=0  && {\rm in}\,[0,T], \\                                                               
	            & u(T, x)=u_T(x)&&{\rm in}\,\,\mathbb{L},  \\                                                                      
	   (\textit{ii})\qquad  & \partial_t m- \partial^2_{xx}\big(\sigma^2(x) m\big)  - \text{div} ( m{q^*}) =0&&{\rm in}\,\,Q, \\    
	            & m(t,0)=0,\,\,\partial_{x}\big(\sigma^2(L) m(t,L)\big) +{q^*}(t,L)m(t,L)=0 && {\rm in}\, [0,T],      \\                                   
	            & m(0,x)=m_0(x) &&{\rm in}\,\,\mathbb{L},                                                        \\
	(\textit{iii})\qquad & {q^*}(t,x)=\arg\max_{q\geq 0} \left\{-q\partial_xu+q\mathrm{P}(t)-\gamma q-\kappa q^2\right\},
\end{aligned}\right.
\end{equation}
where $Q:=(0,T)\times \mathbb{L}$. $P(t,\cdot)$ is the price function. {\cref{MFG} $(i)$ is a backward in time Hamilton-Jacobi-Bellman (HJB) equation characterizing the value function $u(t,x)$ with a terminal value $u_T$ for a representative agent (producer); ($ii$) is a forward in time Fokker-Planck-Kolmogorov equation (FPK) which governs the evolution of $m(t)$ driven by the optimal control ${q^*}(t,x)$; ($iii$) is a fixed point equation characterizing the optimal policy  ${q^*}(t,x)$.} Note that the coupling term in ($i$) depends not only on the distribution of the agents $m(t,x)$, but also on the optimal rate of production ${q^*}(t,x)$. \par
Given the boundary conditions for $u$ and $m$ in the MFG system, the operator defining the FPK equation is the adjoint of the infinitesimal generator of the stochastic process $X_t$, in the sense {that for any function $\phi \in \mathcal C^2(0,L)$ and $\phi(L)=0$ we have}
{
\begin{equation}\label{adjoint}
	\int_0^L\lc- \partial^2_{xx}\lc \sigma^2(x) m\rc  - \text{div} ( m{q^*}) \rc \phi dx=\int_0^L\lc -\sigma^2(x) \partial^2_{xx}\phi+{q^*}\partial_x\phi \rc mdx.
\end{equation}
}
From $(i)$ and $(iii)$ in system~\eqref{MFG}, the equation for the optimal policy can be reformulated more explicitly as
\begin{equation}\label{controlrecursive}
{q^*}(t,x)= \lc \frac{P\big(t,\int_0^L {q^*}(t,x)m(t,x)dx\big)-\gamma-\partial_xu(t,x)}{2\kappa} \rc_+.
\end{equation}
Therefore ${q^*}$ is characterized as a \textit{recursive competitive equilibrium}.\par
Throughout the paper we will be also using the notation for the Hamiltonian:
\begin{equation}\label{H def}
	H\lc\mathrm{P}\lc t\rc-\gamma-\partial_xu\rc=\sup_{q\geq 0}\left\{-q\partial_xu+q\mathrm{P}(t)-\gamma q-\kappa q^2\right\}.
\end{equation}
We denote the flux at $(t,x)$ and aggregate production at $t$, respectively by:
\begin{equation}\label{def w}
	{w^*}(t,x)={q^*}(t,x)m(t,x),\,\,\,\psi(t)=\int_0^L{w^*}(t,x)dx.
\end{equation}
Following \cite{gueant2010mean,MR2762362},   the price function $P$ is determined according to a rule of supply/demand equilibrium on the market, where the demand is given by a function $D(t, P)$,  monotonically decreasing w.r.t. the $P$-variable.
Hence for any aggregate production at time $t$, designated by $a\geq 0$,
\begin{equation}\label{def P}
	D(t,P)=a,\,\,\, P\big(t,a\big)=D(t,\cdot)^{-1}(a).
\end{equation}
and we denote
$
	\mathrm{P}(t)=P\big(t,\psi(t)\big)=D(t,\cdot)^{-1}\big(\psi(t)\big).
$
Two typical examples of demand and price functions are given by (see \cite{MR2762362}):
\begin{align}
\label{DP1}
	D(t,P)=Ee^{\rho t}(\pi_{\text{sub}}-P),\,\,P(t,a)=\pi_{\text{sub}}-\frac{1}{E}e^{-\rho t}a;\\
\label{DP2}
	D(t,P)=Ee^{\rho t}P^{-\eta}-\delta,\,\,P(t,a)=E^{\frac{1}{\eta}}e^{\frac{\rho t}{\eta}}\frac{1}{\big(\delta+a\big)^{\frac{1}{\eta}}}.
\end{align}

In both of these examples, $E>0$ stands for a global wealth factor and $\rho$  for the growth of demand in the economy. In \cref{DP1} the constant $\pi_{\text{sub}}>0$ is exogenous and represents the price of a substitute. The demand in \cref{DP2} satisfies constant elasticity of substitution (CES), where a  high $\delta$ stands for a low price of substitution.\par
{The system \cref{MFG} was proposed in \cite{gueant2010mean,MR2762362} as a continuous time stochastic generalization to the Hotelling's rent theory. It was also discussed in section 1.4.4 of \cite{MR3752669}. A similar type of model, but without considering the demand growth or production cost, was proposed by Chan and Sircar  \cite{MR3359708,chan2017fracking}. Theoretical analysis of the Chan and Sircar model was developed in  \cite{MR3755719,MR4223351,MR3888969,MR4064472,MR4506079}, usually assuming some smallness conditions if the price function is nonlinear  (see \cref{compare with Sircar} for more details). In \cite{bonnans2021schauder,kobeissi2022mean,MR4387201,MR4659381}, MFG of controls with general nonlinear price functions under monotonicity assumptions were considered.} {However, the space domain in models of \cite{bonnans2021schauder,kobeissi2022mean,MR4387201,MR4659381} were either the torus or the whole space. This is not suitable for studying production competition of \textit{exhaustible} resources, which implies a lower bound on the state. } \par
{Under monotonicity assumptions on $P(t,\cdot)$ similar to \cite{bonnans2021schauder,kobeissi2022mean}, which in particular include  the ones defined in \cref{DP1,DP2}, we prove existence and uniqueness of the classical solution to system~\eqref{MFG}. Existence is obtained by proving  the global convergence of   the {\it smoothed policy iteration} (SPI) algorithm, an iterative procedure introduced in \cite{tang2023learning} for solving potential MFGs. SPI is a learning algorithm as fictitious play (cfr. \cite{MR3608094,MR4659381}), but instead of computing an optimal control at each iteration, the agent only evaluates a policy (which amounts to solving a linear problem) and performs a greedy update. The convergence analysis of SPI for Cournot MFG of controls, compared to \cite{tang2023learning}, faces additional difficulties due to the nonlinear fixed point equation \cref{MFG} ($iii$) for characterizing the optimal policy. }\par
Numerical simulations of Cournot MFG of controls have been considered in \cite{MR2762362} and then in \cite{MR3359708} and \cite{MR4659381}. Dealing with the numerical resolution of MFG systems by finite difference method typically requires using some iterative algorithms, e.g. fixed point~\cite{MR4146720,MR4223355}, Newton method~\cite{MR2679575,MR2888257}, fictitious play \cite{MR3608094,MR4030259,MR4368188,MR4659381}, policy iteration \cite{MR4291367,MR4392286,MR4534442} or smoothed policy iteration \cite{tang2023learning}. See e.g.~\cite{MR4214777,MR4368188} and the references therein for more details. In particular, the authors of \cite{MR4659381} considered convergence rate of a fictitious play algorithm for solving an MFG of controls, similar to our Cournot model but {with periodic boundary conditions}. The authors of \cite{MR4146720} considered numerical solution of a Cucker-Smale type MFG of controls with periodic or Neumann boundary conditions. Fictitious play type iterative computation has been considered in \cite{deschamps1975algorithm,MR1039721} for one shot Cournot games with a finite number of agents. {In this paper, we use SPI to solve system \eqref{MFG} with \textit{fully implicit} finite difference schemes. \textit{Implicit} schemes are  advantageous when solving models with a long time horizon and small diffusion. SPI is particularly compatible with implementing \textit{implicit} schemes as equations are linearized at the continuous level.} \par
We observe that, in contrast to \cite{MR2762362} where the Cournot MFG of controls was considered on the interval $[0,T]\times [0,+\infty)$, we study the problem on the bounded interval $[0,T]\times [0,L]$ and we impose a Neumann boundary condition for $u$  and the dual   flux boundary condition  for $m$  at $x=L$. Moreover in \cite{MR2762362}, the dynamics of the agent is given by a geometric Brownian motion (GBM), hence the diffusion coefficient $\sigma^2(x)$ may degenerate at $x=0$. We assume $\sigma>0$ in the theoretical analysis. However, our numerical experiments show that the method performs well in models with GBM diffusions. We plan to consider the theoretical aspects of the degenerate diffusion in the future. We conclude observing that our analytical and numerical method can be extended to consider a wide range of competitive equilibrium models.

\section{Assumptions and preliminary results}\label{Preliminaries and Assumptions}

We first introduce some functional spaces. For any $\mathrm{r}\geq1$, $L^{\mathrm{r}}(\mathbb{L})$ and $L^{\mathrm{r}}(Q)$ are the usual Lebesgue spaces. {Given a Banach space $X$, we denote with $L^{\mathrm{r}}(0,T;X)$ the space of functions $\mathrm{u}:(0,T)\rightarrow X$ such that $\int_0^T\|\mathrm{u}(t)\|^{\mathrm{r}}_Xdt<+\infty$.}
We denote with $H^1(\mathbb{L})$ the space of all functions $\mathrm{u}$ on $\mathbb{L}$, such that $\mathrm{u}(0)=0$ and the norm
$
	\|\mathrm{u}\|_{H^1(\mathbb{L})}:=\|\mathrm{u}\|_{L^{2}(\mathbb{L})}+\|\partial_x\mathrm{u}\|_{L^{2}(\mathbb{L})}
$
is finite. The dual space of $H^1(\mathbb{L})$ is denoted by $H^{-1}(\mathbb{L})$.  We denote by $W^{1,2}_{\mathrm{r}}(Q)$ the space of functions $\mathrm{u}$ such that $\mathrm{u},\partial_t\mathrm{u},\partial_x\mathrm{u}, \partial^2_{xx}\mathrm{u}$ are all in $L^{\mathrm{r}}(Q)$, endowed with the norm
$
	\|\mathrm{u}\|_{W^{1,2}_{\mathrm{r}}(Q)}:=\|\mathrm{u}\|_{L^{\mathrm{r}}(Q)}+\|\partial_t\mathrm{u}\|_{L^{\mathrm{r}}(Q)}+\|\partial_x\mathrm{u}\|_{L^{\mathrm{r}}(Q)}+\|\partial^2_{xx}\mathrm{u}\|_{L^{\mathrm{r}}(Q)}.
$
{We also denote by $V^{0,1}_2(Q)$ the Banach space of functions $\mathrm{u}: Q\rightarrow \mathbb{R}$ such that the norm (cfr. \cite[p. 6]{MR0241822})
$
\|\mathrm{u}\|_{V^{0,1}_2(Q)}:=\max_{0\leq t\leq T}\|\mathrm{u}(t,\cdot)\|_{L^{2}(\mathbb{L})}+\|\partial_x\mathrm{u}\|_{L^{2}(Q)}
$
is finite. }Let $\mathcal{C}^0(Q)$ be the space of continuous functions in $Q$ and $\mathcal C^{0,1}(Q)$ the space of functions in $\mathcal C^0(Q)$ and once continuously differentiable w.r.t. $x$. For $\alpha \in (0,1)$, $\mathcal C^{\alpha/2,\alpha}(Q)$ and $\mathcal C^{1+\alpha/2,2+\alpha}(Q)$ denote the spaces of H\"{o}lder continuous functions in $Q$, endowed with the norms
\begin{align*}
	\|\mathrm{u}\|_{\mathcal C^{\alpha/2,\alpha}(Q)}&:=\|\mathrm{u}\|_{\mathcal C^0(Q)}+\sup_{(t_1,x_1)\neq (t_2,x_2)\in Q}\frac{\vert \mathrm{u}(t_1,x_1)-\mathrm{u}(t_2,x_2)\vert}{\big(\vert x_1-x_2\vert^2+\vert t_1-t_2\vert \big)^{\alpha/2}},\\
	\|\mathrm{u}\|_{\mathcal C^{(1+\alpha)/2,1+\alpha}(Q)}&:=\|\mathrm{u}\|_{\mathcal C^0(Q)}+\|\frac{\partial \mathrm{u}}{\partial x}\|_{\mathcal C^{\alpha/2,\alpha}(Q)}+\sup_{(t_1,x)\neq (t_2,x)\in Q}\frac{\vert \mathrm{u}(t_1,x)-\mathrm{u}(t_2,x)\vert}{\vert t_1-t_2\vert^{(1+\alpha)/2}},\\
	\|\mathrm{u}\|_{\mathcal C^{1+\alpha/2,2+\alpha}(Q)}&:=\|\mathrm{u}\|_{\mathcal C^0(Q)}+\|\frac{\partial \mathrm{u}}{\partial x}\|_{\mathcal C^{(1+\alpha)/2,1+\alpha}(Q)}+\|\frac{\partial \mathrm{u}}{\partial t}\|_{\mathcal C^{\alpha/2,\alpha}(Q)}.
\end{align*}
Throughout the paper we use $r$ for an arbitrary constant such that {$r> 3$}. It is known from \cite[Corollary pp. 342-343]{MR0241822} 
{
\begin{equation}\label{Sobolev}
	W^{1,2}_{r}(Q)\,\,\text{is continuously embedded in}\,\,\mathcal C^{(1+\alpha)/2,1+\alpha}(Q),\quad \alpha=1-\frac{3}{r}.
\end{equation}
}

	\begin{assumption}\label{Main assumptions} We make the following standing assumptions:\\
	\begin{enumerate}[label=\normalfont{\textbf{(A.\arabic*)}},align=left]
		\item\label{assumption1} $m_0$ is $\mathcal C^3(\mathbb{L})$, $\int_0^Lm_0(x)dx=1$, {$m_0(0)=0$ and $m_0(x)>0$ for all $x\in (0,L]$}. 
		\item\label{assumption2} $u_T$ is $\mathcal{C}^3(\mathbb{L})$ and
	      $$
		      \partial_xu_T(x)\geq 0,\,u_T(0)=0,\,\partial_xu_T(L)=0.
	      $$
		\item\label{assumption3} The price function $P(t,a): [0,T]\times \mathbb{R}^+\rightarrow \mathbb{R}$ is $\mathcal{C}^{2}$ for both variables. \\$\partial_aP(t,a)\leq 0$ for all $ t\in [0,T]$. There exists $\Phi(t,a): [0,T]\times \mathbb{R}^+\rightarrow \mathbb{R}$, $\Phi$ being $\mathcal C^3$ for both variables, such that $P(t,a)=\partial_a\Phi(t,a)$ for all $ t\in [0,T]$. Moreover, there exist $C_P>0$, such that $\gamma<P(t,0)\leq C_P+\gamma$, for all $t\in [0,T]$.
	     
		\item\label{assumption4} The coefficient $\sigma$ is $\mathcal C^2(\mathbb{L})$ and   non degenerate, i.e. $\sigma(x)\ge \sigma_0>0$.
	\end{enumerate}
	\end{assumption}	
\begin{remark}\label{J}
{With \ref{assumption3}, we can define the potential formulation
\begin{equation*}
J(q,m)=\int_0^T e^{-\lambda t}\Phi(t,\int_0^L qmdx)dt-\int_Qe^{-\lambda t} \lc\gamma q+\kappa q^2\rc mdxdt.
\end{equation*}
It can be shown with a method similar to \cite{bonnans2021schauder} that system \cref{MFG} can be derived from the first order condition of maximizing $J(q,m)$ with $(q,m)$ constrained to satisfy \cref{MFG} $(ii)$. The proof will be different with the absorption boundary condition. As one can use similar methods from \cite{MR4223355}, we omit the details. $J(q,m)$ will be used as an auxiliary functional in our main convergence result. }
\end{remark}
\begin{remark}
	\ref{assumption3} implies
	\begin{equation}\label{P bound}
		P(t,a)\leq C_P+\gamma, \quad \forall t\in [0,T].
	\end{equation}
	We observe that \ref{assumption3} is satisfied by both examples \cref{DP1} and \cref{DP2}, by setting $C_P=\pi_{\text{sub}}-\gamma$ in \cref{DP1} and $C_P=(\frac{E}{\delta})^{\frac{1}{\eta}}-\gamma$ in \cref{DP2}. 
\end{remark}
\begin{remark}\label{compare with Sircar}
	We now consider a particular case of the linear price impact model \cref{DP1} by taking $\rho=\gamma=0$ and $\kappa=1$. The revenue of the producer (income minus cost) then becomes:
		$
			q\mathrm{P}(t)-\gamma q-\kappa q^2=q(\pi_{\text{sub}}-\frac{1}{E}\psi(t)-q).
		$
		If we interpret $\frac{1}{E}$ as the measure of competition between producers and $\tilde{P}(t)=\pi_{\text{sub}}-\frac{1}{E}\psi(t)-q$ as the price, then we recover the price impact model in literature \cite{MR4064472,MR3755719,MR3888969,MR4223351,MR4506079}. Generalizations, e.g. in \cite{MR4506079}, consist in price of the form
	$
		\tilde{P}(t)=\pi_{\text{sub}}-\Psi \big(\frac{1}{E}\psi(t)+q\big),
	$
	where  $\Psi$ is a nonlinear monotone increasing function. A typical condition for the existence and uniqueness of solution to the Cournot MFG system  with price defined by $\tilde{P}(t)$ is the smallness of $\frac{1}{E}$ or $T$. Smallness conditions for  general MFG of controls model have been discussed in \cite{MR4565014,MR4387201}. In our model, following \cite{MR2762362}, we assume that the price depends only on the average production $\psi(t)$ and we make no smallness assumptions on the parameters throughout the paper. {A key assumption is the monotonicity of $P(t,\cdot)$, which follows closely previous literature \cite{kobeissi2022mean} and \cite{bonnans2021schauder}. However, our model differs from \cite{kobeissi2022mean} and \cite{bonnans2021schauder} because of the boundary condition and control constraint $q\geq 0$.}
\end{remark}

We first consider some results related to the FPK equation with Robin boundary conditions (see \cite{MR4223355,MR3888965,MR3333058} for related results). Given the initial-boundary value problem
\begin{equation}\label{FPK}
	\left\{\begin{aligned}
		  \qquad &	\partial_t m- \partial^2_{xx}(\sigma^2(x)m)-{\rm{div}} (m q)=0 \qquad  &&{\rm in}\,Q,    \\
		  \qquad & m(t,0)=0\,\,\text{and} \,\,\partial_x\big(\sigma^2(L) m(t,L)\big)+m(t,L) q(t,L)=0 && {\rm in}\, [0,T],                    \\
		  \qquad & m(0,x)=m_0(x)                                                                      && {\rm in}\,\mathbb{L}.
	\end{aligned}\right.
\end{equation}
We give the following definition of weak solution (see \cite[Chapter V, section 1]{MR0241822}).
\begin{definition}\label{def weak sol}
	A function $m$ is said to be a bounded weak solution to \cref{FPK} if $m\in L^\infty(Q)$,  $\partial_x m\in L^2(Q)$ and
	\begin{equation*}
	\int_0^L m(t,x)\varphi(t,x) dx-\int_0^L m_0(x)\varphi(0,x) dx+\int_0^t\int_0^L\left(-m\partial_t \varphi+\partial_x\varphi \partial_x\big(\sigma^2(x)m\big)+q m\partial_x\varphi  \right)dxdt=0,
	\end{equation*}
	for any $t\in (0,T]$ and any function $\varphi$ such that $\varphi$, $\partial_t\varphi, \partial_x\varphi\in L^2(Q)$, $\varphi(t,0)=\varphi(t,L)=0$.
\end{definition}
\begin{lemma}\label{m stability}
	Suppose $m_0$ satisfies \ref{assumption1} and $0\leq q(t,x)\leq R$.
	Then, there exists a unique solution $m$ to \cref{FPK} satisfying $m\in L^2(0,T;H^1(\mathbb{L}))\cap L^\infty(Q)$, $\partial_tm\in L^2\big(0,T;H^{-1}(\mathbb{L})\big)$. In addition, $\|m\|_{C^{\alpha/2,\alpha}}(Q)\leq C$, {$m(t,x)\geq 0$ in $Q$ and $m(t,x)>0$ for all $(t,x)\in [0,T]\times (0,L)$}.
	Let $m_\iota$, $\iota=\{1,2\}$, be the solution to \cref{FPK} corresponding to $q_\iota$ and set
	$\varrho=m_1-m_2$. Then
	$\|\varrho\|_{L^\infty(0,T;L^2(\mathbb{L}))}\leq C_\varrho \|q_1-q_2\|_{L^\infty(Q)}$, 
	where $C_\varrho$ is a constant depending only on {$m_0$, $T$, $L$, $R$, $\|\sigma\|_{\mathcal C^1(\mathbb{L})}$ and $\sigma_0$}. 
\end{lemma}

\begin{proof}  The existence and uniqueness of a weak solution
$m\in L^2(0,T;H^1(\mathbb{L}))$, $\partial_tm\in L^2\big(0,T;H^{-1}(\mathbb{L})\big)$
is standard from energy estimates (cfr. \cite{MR4392286}). From \cite[Proposition 2.2]{MR3888965}, there exists a constant $C_m$ which depends only on $m_0$, $L$, $T$ and $R$ such that
$
	\|m\|_{L^\infty(Q)}\leq C_m.
$
Therefore, $m$ is a bounded weak solution in the sense of \cref{def weak sol} and $\|m\|_{C^{\alpha/2,\alpha}(Q)}\leq C$ follows from \cite[Chapter V, Theorem 1.1]{MR0241822}. {One can use duality method to show $m(t,x)\geq 0$ as in \cite[Proposition 4.3]{MR3559742}. Suppose $m(t^*,x^*)=0$, $(t^*,x^*)\in (0,T]\times (0,L)$. From $m(0,x^*)>0$ and $\|m\|_{C^{\alpha/2,\alpha}(Q)}\leq C$, there exists $\tau_0$ sufficiently small such that $m(\tau_0,x^*)\geq \frac{1}{2}m(0,x^*)>0$. On the other hand, from  $m(t^*,x^*)=0$ and by applying Harnack inequality (e.g. \cite[Theorem 1.1]{GYONGY2024857}), we have  $m(\tau_0,x^*)=0$, a contradiction. Hence $m(t,x)>0$ for all $(t,x)\in [0,T]\times (0,L)$.} It is clear that $\varrho=m_1-m_2$ satisfies 
\begin{equation*}
	\left\{\begin{aligned}
		\qquad & \partial_t\varrho-\partial^2_{xx}\big(\sigma^2(x) \varrho\big)-{\rm{div}} (\varrho q_2)={\rm{div}} (m_1(q_1-q_2)),    \\
		\qquad & \partial_x\big(\sigma^2(L)\varrho(t,L)\big)+\varrho(t,L) q_{2}(t,L)+m_1(t,L)(q_1(t,L)-q_2(t,L))=0,                    \\
		\qquad & \varrho(t,0)=0, \,\varrho(0,x)=0.
	\end{aligned}\right.
\end{equation*}
{By energy estimates (cfr. \cite{MR4392286} or \cite[Lemma 17]{MR4659381})}, we obtain the bound for $\varrho$.
\end{proof}
\begin{corollary}\label{m L1}
	Under the assumptions of \cref{m stability}, for all $0\leq t_0<t_1\leq T$, $\int_0^Lm(t_1,x)dx\leq \int_0^Lm(t_0,x)dx$.
\end{corollary}
\begin{proof}
Since $\sigma^2(0) m(t,0)=0$ and $\sigma^2(x) m(t,x)\geq 0$ for all $x\in (0,L]$, then $\partial_x\big(\sigma^2(0) m(t,0)\big)\geq 0$. We obtain
$$
	\int_0^Lm(t_1,x)dx-\int_0^Lm(t_0,x)dx=-\int_{t_0}^{t_1}\partial_x\big(\sigma^2(0) m(t,0)\big)dt\leq 0
$$
from integration by parts.
\end{proof}

We now consider some preliminary results. Define the space 
{
\begin{equation}\label{bbmX}
\mathbbm{X}:=\big\{(u,m,q):\,\,u\in W^{1,2}_r(Q), m\in L^2(0,T;H^1(\mathbb{L}))\cap L^\infty(Q),\,\,\partial_tm\in L^2\big(0,T;H^{-1}(\mathbb{L})\big), q\in L^\infty(Q) \big\},
\end{equation}
for functions $u,m: Q\rightarrow \mathbb{R}$ and $q: Q\rightarrow \mathbb{R}^+\cup \{0\}$.} We call a triple $(u,m,{q^*})\in \mathbbm{X}$ a weak solution to the Cournot MFG of controls system \eqref{MFG} if (\textit{i}) and (\textit{ii}) are satisfied in the sense of distributions and (\textit{iii}) is satisfied pointwise.

\begin{lemma}\label{monotone}
	Suppose $(u,m,{q^*})$ is a weak solution to system~\eqref{MFG}. Then $u(t,x)\geq 0$ and 
	$\partial_xu(t,x)\geq 0$. In addition,
	\begin{equation}\label{q bound}
		0\leq {q^*}(t,x)\leq \frac{C_P}{2\kappa} \quad  \text{and}\quad
		\mathrm{P}(t)> \gamma\quad  \quad \forall (t,x)\in Q.
	\end{equation}
\end{lemma}
\begin{proof}
It is clear from \ref{assumption2} that $u_T(x)\geq 0$ for all $x\in [0,L]$. From the HJB equation in system~\eqref{MFG} obviously
$
	\sup_{q\geq 0}\Big\{-q\partial_xu+q\mathrm{P}(t)-\gamma q-\kappa q^2\Big\}\geq 0,
$
hence
$
	-\partial_tu-\sigma^2(x) \partial^2_{xx}u+\lambda u\geq 0.
$
The weak parabolic maximum principle implies that $u(t,x)\geq 0$. Since $u(t,0)=0$ for all $t$ we obtain $\partial_xu(t,0)\geq 0$ for all $t$. From \cref{P bound}, $\mathrm{P}\lc t\rc-\gamma\leq C_P$. Recall that $(\cdot)_+$ is a convex function, hence the Hamiltonian
\begin{equation*}
	H\lc\mathrm{P}\lc t\rc-\gamma-\partial_xu\rc=\kappa \lc \frac{\mathrm{P}\lc t\rc-\gamma-\partial_xu}{2\kappa}\rc_+^2
	\leq \frac{1}{\kappa}\lc \frac{(\mathrm{P}\lc t\rc-\gamma)_++(-\partial_xu)_+}{2}\rc^2
	\leq \frac{C_P^2+(\partial_x u)^2}{2\kappa}.
\end{equation*}
It follows from standard estimates for quasilinear parabolic equations, e.g. \cite[Chapter 5, Theorem 6.3]{MR0241822}, that $\partial_xu\in L^\infty(Q)$ and $u\in W^{1,2}_r(Q)$. The derivative of $H\lc\mathrm{P}\lc t\rc-\gamma-\partial_xu\rc$ w.r.t $x$ in the sense of distributions is
$$
	\lc \frac{\mathrm{P}\lc t\rc-\gamma-\partial_xu}{2\kappa}\rc_+\mathbf{1}_{\{\mathrm{P}\lc t\rc-\gamma-\partial_xu\geq 0\}}(-\partial_{xx}^2u),
$$
where $\mathbf{1}_{\{\mathrm{P}\lc t\rc-\gamma-\partial_xu\geq 0\}}$ denotes an indicator function. We can differentiate \cref{MFG} ($i$), with $\phi=\partial_xu$ being a weak solution to
$$
	\partial_{t}\phi+\sigma^2(x) \partial_{xx}^2\phi+2\sigma^2(x)\sigma'(x)\partial_{x}\phi-\lambda \phi- \lc \frac{\mathrm{P}\lc t\rc-\gamma-\phi}{2\kappa}\rc_+\mathbf{1}_{\{\mathrm{P}\lc t\rc-\gamma-\phi\geq 0\}}\partial_{x}\phi=0.
$$
From $\partial_xu\in L^\infty(Q)$, it is clear that $- \lc \frac{\mathrm{P}\lc t\rc-\gamma-\phi}{2\kappa}\rc_+\mathbf{1}_{\{\mathrm{P}\lc t\rc-\gamma-\phi\geq 0\}}$ is bounded in $L^\infty(Q)$. It is easy to see from $u\in W^{1,2}_r(Q)$ (or by using energy estimates) that $\phi \in V^{0,1}_2(Q)$. From $\phi(T,x)\geq 0$, $\phi(t,L)=0$ and $\phi(t,0)\geq 0$, we can use the parabolic maximum principle for weak solutions (cfr. \cite[Chapter 3, Theorem 7.2, p. 188]{MR0241822}) and obtain $\phi(t,x)\geq 0$ for all $(t,x)\in Q$. From $\partial_xu(t,x)\geq 0$ and \cref{controlrecursive} we obtain \cref{q bound}.\par
Next, {we show $\mathrm{P}(t)> \gamma$} by a contradiction argument. Suppose $\mathrm{P}(\check{t})\leq \gamma$ at some time $\check{t}$. Then from  $\partial_xu(\check{t},x)\geq 0$ we have
$
	\mathrm{P}(\check{t})-\gamma-\partial_xu(\check{t},x)\leq 0.
$
With  \cref{controlrecursive}, we can then obtain ${q^*}(\check{t},x)=0$ for all $x\in [0,L]$. This implies
$$
	P(\check{t},0)=P\lc \check{t},\int_0^L{q^*}(\check{t},x)m(\check{t},x)dx\rc=\mathrm{P}(\check{t}) \leq \gamma.
$$
This is a contradiction as we have assumed $P(t,0)>\gamma$ for all $t\in [0,T]$ in \ref{assumption3}.
\end{proof}
\begin{remark}\label{remark P}
	\cref{monotone} shows that the price $\mathrm{P}(t)$ will never be lower than $\gamma$. It is clear that for $(t,x)$ such that ${q^*}(t,x)>0$,
	$$
		\partial_xu(t,x)=\mathrm{P}(t)-\frac{d\left(\gamma {q^*}+\kappa {q^*}^2\right)}{d{q^*}}.
	$$
	In other words, $\partial_xu$ equals the price minus the marginal production cost. Therefore, $\partial_xu$ can be interpreted as the Hotelling rent in \cite{MR2762362}, an index  widely used for measuring scarcity of resources in the economic literature. Hence, another feature of the model is that the Hotelling rent   is  non negative.
\end{remark}
\begin{remark}\label{remark q}
	From \cref{monotone}, we can in fact replace the constraint $q\geq 0$ in \cref{MFG} ($i$) and ($iii$) by $0\leq q\leq \frac{C_P}{2\kappa}$, 
	$
		{q^*}= \arg \max_{0\leq q\leq \frac{C_P}{2\kappa}} \left\{-q\partial_xu+q\mathrm{P}(t)-\gamma q-\kappa q^2\right\}.
	$
	This implies that the Hamiltonian is globally Lipschitz and therefore allows us to avoid using Bernstein estimates as in \cite{MR4387201}.
\end{remark}\par
The following result is crucial for proving uniqueness of the solution to system \eqref{MFG} and convergence of learning algorithms. It is analogous to \cite[Lemma 2.7]{MR3608094} and \cite[Lemma 2.8]{tang2023learning}. Here, since the admissible set of controls $q$ is compact( cfr. \cref{remark q}) we give a modified proof. The novelty in the proof is that we consider separately the case when the constraints on $q$ are binding. 
\begin{lemma}\label{H-L}
	Let $\Lambda\in \mathbb{R}$, 
	$$
	H(\Lambda)=\sup_{0\leq q\leq \frac{C_P}{2\kappa}}\left\{-q\Lambda-\kappa q^2\right\},\,
	 {q^*}=\arg \max_{0\leq q\leq \frac{C_P}{2\kappa}}\left\{-q\Lambda-\kappa q^2\right\},
	 $$ 
	then for all $q$ such that $0\leq q\leq \frac{C_P}{2\kappa}$,
	\begin{equation}\label{H-Lgeq}
		H(\Lambda)-\Big(-q\Lambda-\kappa q^2\Big)\geq \kappa |{q^*}-q|^2.
	\end{equation}
\end{lemma}
\begin{proof}
We have: 
$
	H(\Lambda)-\Big(-q\Lambda-\kappa q^2\Big)=\kappa q^2-\kappa {q^*}^2-\big(-\Lambda(q-{q^*})\big).
$\par
If $-C_P\leq \Lambda\leq 0$, then ${q^*}=-\frac{\Lambda}{2\kappa}$, hence \cref{H-Lgeq}. If $\Lambda> 0$, then ${q^*}=0$,
$$
	H(\Lambda)-\Big(-q\Lambda-\kappa q^2\Big)=\kappa q^2+\Lambda q\geq \kappa |0-q|^2.
$$
If $\Lambda< -C_P$, then ${q^*}=\frac{C_P}{2\kappa}$ and $q-{q^*}\leq 0$,  hence $-\Lambda(\frac{C_P}{2\kappa}-q)\geq C_P(\frac{C_P}{2\kappa}-q)$,
$$
	H(\Lambda)-\Big(-q\Lambda-\kappa q^2\Big)=-\Lambda(\frac{C_P}{2\kappa}-q)+\kappa q^2-\kappa \left(\frac{C_P}{2\kappa}\right)^2
	\geq \kappa|\frac{C_P}{2\kappa}-q|^2.
$$
\end{proof}

\section{Uniqueness of the solution}
In this section, we prove uniqueness of the solution to system~\eqref{MFG}. {We use the standard method for system satisfying Lasry-Lions monotonicity conditions, which has been used in \cite[Proposition 2]{bonnans2021schauder} for considering a Cournot MFG of controls with periodic boundary condition.} \par

\begin{theorem}\label{uniqueness}
	Under assumptions \ref{assumption1}, \ref{assumption2} and \ref{assumption3}, system \eqref{MFG} has at most one weak solution
	$(u,m,{q^*})\in \mathbbm{X}$.
\end{theorem}
\begin{proof}
Suppose there exist two solutions $(u_1,m_1,{q^*_1})$, $(u_2,m_2,{q^*_2})$. Let $v=u_1-u_2$, $\varrho=m_1-m_2$. For $\iota \in \{1,2\}$, we set ${w^*_\iota}={q^*_\iota} m_\iota$ and $\mathrm{P}_\iota(t)=P\big( t,\int_0^L {w^*_\iota} dx\big)$. Hence
$$
	H\big(\partial_xu_\iota-\mathrm{P}_\iota(t)+\gamma\big)=-{q^*_\iota}\partial_xu_\iota+{q^*_\iota} \mathrm{P}_\iota (t)-\gamma {q^*_\iota}-\kappa {q^*_\iota}^2.
$$
From \cref{H-L},   setting $\Lambda=\partial_xu_2-\mathrm{P}_2(t)+\gamma$ and  $\Lambda=\partial_xu_1-\mathrm{P}_1(t)+\gamma$ in \cref{H-Lgeq}, we obtain
\begin{align}\label{H-Lgeq1}
H\big(\partial_xu_2-\mathrm{P}_2(t)+\gamma\big)-\Big(-{q^*_1}\big(\partial_xu_2-\mathrm{P}_2(t)+\gamma\big)-\kappa {q^*_1}^2\Big)\geq \kappa |{q^*_1}-{q^*_2}|^2,\\
\label{H-Lgeq2}
H\big(\partial_xu_1-\mathrm{P}_1(t)+\gamma\big)-\Big(-{q^*_2}\big(\partial_xu_1-\mathrm{P}_1(t)+\gamma\big)-\kappa {q^*_2}^2\Big)\geq \kappa |{q^*_1}-{q^*_2}|^2.
\end{align}
By comparing $(i)$ and $(ii)$ in the systems for $(u_1,m_1,{q^*_1})$ and $(u_2,m_2,{q^*_2})$, we get
{
\begin{equation}\label{v eqn}
\left\{\begin{aligned}
(i)\quad	&\partial_tv+\sigma^2(x) \partial^2_{xx}v -\lambda v+H\big(\partial_xu_1-\mathrm{P}_1(t)+\gamma\big)-H\big(\partial_xu_2-\mathrm{P}_2(t)+\gamma\big)=0,\\
(ii)\quad		&\partial_t\varrho-\partial^2_{xx}\big(\sigma^2(x) \varrho\big)-\partial_x({w^*_1}-{w^*_2})=0.
\end{aligned}\right.
\end{equation}
Multiplying \cref{v eqn} $(i)$ and $(ii)$ by $e^{-\lambda t}\varrho$ and  $e^{-\lambda t}v$ respectively}, adding the resulting equations and integrating over $Q$, we get
\begin{equation*}
	\begin{aligned}
		0={}& \int_Qe^{-\lambda t}\varrho\Big(\partial_tv+\sigma^2(x) \partial^2_{xx}v - \lambda v \Big)dxdt +\int_Qe^{-\lambda t}\varrho \Big(H\big(\partial_xu_1-\mathrm{P}_1(t)+\gamma\big)-H\big(\partial_xu_2-\mathrm{P}_2(t)+\gamma\big)\Big)dxdt \\
		     & +\int_Qe^{-\lambda t}v\Big( \partial_t\varrho-\partial^2_{xx}\big(\sigma^2(x) \varrho\big)-\partial_x({w^*_1}-{w^*_2})\Big)dxdt.
	\end{aligned}
\end{equation*}
Integrating by parts, we obtain
\begin{equation}\label{eq_diff}
	\begin{aligned}
	& -\int_Qe^{-\lambda t}m_1\Big(H\big(\partial_xu_2-\mathrm{P}_2(t)+\gamma \big)-H\big(\partial_xu_1-\mathrm{P}_1(t)+\gamma\big)-q^*_1(\partial_xu_1-\partial_xu_2)\Big)dxdt  \\
		     & -\int_Qe^{-\lambda t}m_2\Big(H\big(\partial_xu_1-\mathrm{P}_1(t)+\gamma \big)-H\big(\partial_xu_2-\mathrm{P}_2(t)+\gamma\big)-q^*_2(\partial_xu_2-\partial_xu_1)\Big)dxdt=0.
	\end{aligned}
\end{equation}
Since
\begin{equation*}
H\big(\partial_xu_1-\mathrm{P}_1(t)+\gamma\big)+q^*_1(\partial_xu_1-\partial_xu_2)
=-q^*_1\big(\partial_xu_2-\mathrm{P}_2(t)+\gamma\big)-\kappa (q^*_1)^2+q^*_1\big(\mathrm{P}_1(t)-\mathrm{P}_2(t)\big),
\end{equation*}
we have
\begin{equation}\label{P1-P2}
	\begin{aligned}
		    & \int_Qe^{-\lambda t}m_1\Big(H\big(\partial_xu_2-\mathrm{P}_2(t)+\gamma\big)-H\big(\partial_xu_1-\mathrm{P}_1(t)+\gamma\big)-{q^*_1}(\partial_xu_1-\partial_xu_2)\Big)dxdt \\
		={} & \int_Qe^{-\lambda t}m_1\Big(H\left(\partial_xu_2-\mathrm{P}_2(t)+\gamma\right)-\big(-{q^*_1}(\partial_xu_2-\mathrm{P}_2(t)+\gamma)-\kappa {q^*_1}^2\big)\Big)dxdt       \\
		    & -\int_0^Te^{-\lambda t}\Big(\big(\int_0^L{w^*_1}dx\big)\big(\mathrm{P}_1(t)-\mathrm{P}_2(t)\big)\Big)dt.
	\end{aligned}
\end{equation}
Likewise
\begin{equation}\label{P2-P1}
	\begin{aligned}
		    & \int_Qe^{-\lambda t}m_2\Big(H\big(\partial_xu_1-\mathrm{P}_1(t)+\gamma\big)-H\big(\partial_xu_2-\mathrm{P}_2(t)+\gamma\big)-{q^*_2}(\partial_xu_2-\partial_xu_1)\Big)dxdt \\
		={} & \int_Qe^{-\lambda t}m_2\Big(H\big(\partial_xu_1-\mathrm{P}_1(t)+\gamma\big)-\big(-{q^*_2}(\partial_xu_1-\mathrm{P}_1(t)+\gamma)-\kappa {q^*_2}^2\big)\Big)dxdt       \\
		    & -\int_0^Te^{-\lambda t}\Big(\big(\int_0^L{w^*_2}dx\big)\big(\mathrm{P}_2(t)-\mathrm{P}_1(t)\big)\Big)dt.
	\end{aligned}
\end{equation}

Combining \cref{H-Lgeq1}-\cref{H-Lgeq2} with \cref{P1-P2} and \cref{P2-P1}, and recalling the definition  of $\mathrm{P}_1(t)$ and $\mathrm{P}_2(t)$, by \cref{eq_diff}  we obtain, denoting $\psi_\iota(t)$ as in \cref{def w},
\begin{equation}
 \kappa \int_Qe^{-\lambda t}(m_1+m_2)|{q^*_1}-{q^*_2}|^2dxdt                                                                                                                                                    
\leq  \int_0^Te^{-\lambda t}\big(\psi_1(t)-\psi_2(t)\big)\big(\mathrm{P}_1(t)-\mathrm{P}_2(t)\big)dt,
\end{equation}
{hence
$
 \big(\psi_1(t)-\psi_2(t)\big)\big(\mathrm{P}_1(t)-\mathrm{P}_2(t)\big) \geq 0                                                         
$
for all $t\in [0,T]$. Meanwhile 
$
 \big(\psi_1(t)-\psi_2(t)\big)\big(\mathrm{P}_1(t)-\mathrm{P}_2(t)\big) \leq 0                                                      
$
from \ref{assumption3}. Therefore $\mathrm{P}_1(t)=\mathrm{P}_2(t)$ a.e. in $[0,T]$. Moreover, it follows from \cref{q bound} and $\partial_xu_\iota \geq 0$ that 
$
0\leq \lc \mathrm{P}_\iota\lc t\rc-\gamma-\partial_xu_\iota\rc_+\leq C_P
$.
From \cref{H def} and $\mathrm{P}_1(t)=\mathrm{P}_2(t)$ a.e. we have
$$
\left\vert H\big(\partial_xu_1-\mathrm{P}_1(t)+\gamma\big)-H\big(\partial_xu_2-\mathrm{P}_2(t)+\gamma\big)\right\vert 
\leq  \frac{C_P\vert \partial_xu_1-\partial_xu_2\vert}{2\kappa},
$$
and therefore by \cref{v eqn} $(i)$,
$$
-\frac{C_P\vert \partial_xv\vert}{2\kappa} \leq -\partial_tv-\sigma^2(x) \partial^2_{xx}v +\lambda v \leq \frac{C_P\vert \partial_xv\vert}{2\kappa}.
$$
From weak maximum principle, we obtain $u_1(t,x)=u_2(t,x)$ for all $(t,x)\in Q$. With $\mathrm{P}_1(t)=\mathrm{P}_2(t)$ a.e. we obtain ${q^*_1}={q^*_2}$ a.e. in $Q$.}
\end{proof}

\section{Learning algorithm  and existence of the solution}\label{sec:learning}
We introduce the \textit{smoothed policy iteration algorithm} (SPI for short), motivated by the   algorithm (SPI1)  introduced in \cite{tang2023learning}. As we have shown that system~\eqref{MFG} has at most one solution in \cref{uniqueness}, we prove the global convergence of the algorithm and this shows the existence of a solution to the system.\\
Given $\bar{q}^{(0)}\in \mathcal C^2(Q)$, iterate for each $n\geq 0$: 
\begin{itemize}
	\item[(i)] \textit{Generate the distribution from the smoothed policy}. Solve
	      \begin{equation}\label{m update 1}
		      \left\{\begin{aligned}
			      \qquad & \partial_t m^{(n)}-\partial^2_{xx}\big(\sigma^2(x)  m^{(n)}\big)-\text{div} (m^{(n)} \bar{q}^{(n)})=0,        \\
			              \qquad & m^{(n)}(t,0)=0,\, \partial_x\big(\sigma^2(L) m^{(n)}(t,L)\big)+\bar{q}^{(n)}(t,L)m^{(n)}(t,L)=0,         \\
			              \qquad & m^{(n)}(0,x)=m_0(x).
		      \end{aligned}\right.
	      \end{equation}
	\item[(ii)] \textit{Price update}. Set
	      \begin{equation}\label{price update}
		      \mathrm{P}^{(n)}(t)=P\big(t,\int_0^L \bar{w}^{(n)}(t,x)dx \big),\,\,\bar{w}^{(n)}=\bar{q}^{(n)}m^{(n)} \quad \forall t\in [0,T].
	      \end{equation}

	\item[(iii)] \textit{Policy evaluation}. Solve
	      \begin{equation}\label{u update 1}
		      \left\{\begin{aligned}
			       \qquad & \partial_t u^{(n)}+\sigma^2(x) \partial^2_{xx}u^{(n)}-\lambda u^{(n)}-\bar{q}^{(n)} \partial_xu^{(n)}
			    +\bar{q}^{(n)} \mathrm{P}^{(n)}(t)-\gamma \bar{q}^{(n)}-\kappa\big(\bar{q}^{(n)}\big)^2=0,        \\
			       & u^{(n)}(t,0)=0,\,\partial_xu^{(n)}(t,L)=0, \,u^{(n)}(T,x)=u_T(x).
		      \end{aligned}\right.
	      \end{equation}
	\item[(iv)] \textit{Policy update}.
	      \begin{equation}\label{q update 1}
		      q^{(n+1)}(t,x)=\min \left\{\lc\frac{\mathrm{P}^{(n)}(t)-\gamma-\partial_xu^{(n)}(t,x)}{2\kappa}\rc_+,\frac{C_P}{2\kappa}\right\}.
	      \end{equation}
	\item[(v)] \textit{Policy smoothing}
	      \begin{equation}\label{average}
		      \bar{q}^{(n+1)}=(1-\zeta_n)\bar{q}^{(n)}+\zeta_nq^{(n+1)}, \quad  \forall (t,x)\in Q.
	      \end{equation}
\end{itemize}
In the following, we assume

\begin{equation}\label{assumption-learning-rates}
	\zeta_n=\frac{\beta}{n+\beta}, \,\,\beta\in \mathbb{N}\,\, \text{for all}\,\, n \ge 0.
	\end{equation}
For $\beta=1$ policy smoothing \cref{average} means $\bar{q}^{(n)}=\frac{1}{n}\sum_{k=1}^n q^{(k)}$. For $\beta=2$, \cref{average} becomes
	$$
		\bar{q}^{(n)}=\frac{1}{\sum_{k=0}^{n}k}\sum_{k=0}^n kq^{(k)}
		= 
		\frac{1}{\frac{1}{2}n(n+1)}q^{(1)}+\frac{2}{\frac{1}{2}n(n+1)}q^{(2)}+\cdots+\frac{n}{\frac{1}{2}n(n+1)}q^{(n)}.
	$$
	i.e. one puts larger weights on more recent observations of $q^{(n)}$ , cfr. \cite[p. 191]{deschamps1975algorithm} and \cite[Remark 3.1]{tang2023learning}.

\begin{remark}
	The rationale of updating policies only in the compact set $0\leq q\leq \frac{C_P}{2\kappa}$ is due to \cref{monotone} and \cref{remark q}.
\end{remark}

\begin{lemma}\label{Solution u m1}
	For each $n\geq 1$, there exists a unique solution $\big(u^{(n)},m^{(n)}\big)$, 
	$$
	u^{(n)}  \in W^{1,2}_r(Q),\,\,m^{(n)}\in L^2(0,T;H^1(\mathbb{L}))\cap L^\infty(Q),\,\,\partial_tm^ {(n)}\in L^2\big(0,T;H^{-1}(\mathbb{L})\big)
        $$
	to the system \cref{m update 1}-\cref{u update 1} such that
	$
		\|u^{(n)}\| _{W^{1,2}_r(Q)}+\|m^{(n)}\|_{C^{\alpha/2,\alpha}(Q)}\leq C,
	$
	where $C$ is independent of $n$.
\end{lemma}
\begin{proof}
From $0\leq q^{(n)}\leq \frac{C_P}{2\kappa}$, we also have $0\leq \bar{q}^{(n)}\leq \frac{C_P}{2\kappa}$. Then the result follows by using \cref{m stability} and standard parabolic estimates.
\end{proof}

\begin{lemma}\label{C/n bound algo1}
	There exists a constant $C$, such that for all $n\geq 1$,
	\begin{equation}
			    \| \bar q^{(n+1)}- \bar q^{(n)}\|_{L^\infty(Q)}+\| q^{(n+1)}- q^{(n)}\|_{L^\infty(Q)}           
			+ \| m^{(n+1)}-m^{(n)}\|_{L^\infty(0,T;L^2(\mathbb{L}))}+\| u^{(n+1)}-u^{(n)}\| _{W^{1,2}_r(Q)}\leq \frac{C}{n}.
	\end{equation}
\end{lemma}
\begin{proof}
{From \cref{q update 1,average}, we have
\begin{equation}\label{qn+1-qn}
\bar{q}^{(n+1)}-\bar{q}^{(n)}=\zeta_n(q^{(n+1)}-\bar{q}^{(n)}),
\end{equation}}
hence
\begin{equation}\label{estim_q}
	\| \bar q^{(n+1)}- \bar q^{(n)}\|_{L^\infty(Q)}\leq \frac{C_P}{\kappa n},\,\,\| \bar q^{(n+1)}+ \bar q^{(n)}\|_{L^\infty(Q)}\leq \frac{C_P}{\kappa},\,\,\| (\bar q^{(n+1)})^2- (\bar q^{(n)})^2\|_{L^\infty(Q)}\leq \frac{C^2_{P}}{\kappa^2 n}.
\end{equation}
We denote
$
	\delta u^{(n+1)}=u^{(n+1)}-u^{(n)},\,\,\delta m^{(n+1)}=m^{(n+1)}-m^{(n)}.
$
From
$$
	\partial_t\delta m^{(n+1)}-\partial^2_{xx}\big(\sigma^2(x) \delta m^{(n+1)}\big)-{\rm{div}}(\bar q^{(n)}\delta m^{(n+1)})={\rm{div}}((\bar q^{(n+1)}-\bar q^{(n)})m^{(n+1)})
$$
and \cref{m stability}, it follows that $\| m^{(n+1)}-m^{(n)}\|_{L^\infty(0,T;L^2(\mathbb{L}))}\leq C/n$. This with \cref{price update} and \cref{estim_q} imply 
$ \|\bar{w}^{(n+1)}-\bar{w}^{(n)}\|_{L^\infty(0,T;L^2(\mathbb{L}))}\leq C/n$. We can obtain from \ref{assumption3}
\begin{equation*}
\sup_{t\in [0,T]}\left\vert \mathrm{P}^{(n+1)}(t)-\mathrm{P}^{(n)}(t) \right\vert \leq \frac{C}{n}\,\text{and}\,
	\sup_{(t,x)\in Q}\left\vert \bar q^{(n+1)} \mathrm{P}^{(n+1)}(t)-\bar q^{(n)}\mathrm{P}^{(n)}(t) \right\vert \leq \frac{C}{n}.
\end{equation*}
From
\begin{align*}
		    & \partial_t \delta u^{(n+1)}+\sigma^2(x) \partial^2_{xx}\delta u^{(n+1)}-\bar{q}^{(n+1)}   \partial_x \delta u^{(n+1)}                                             \\
		={} & -(\bar{q}^{(n+1)}-\bar{q}^{(n)})\partial_x  u^{(n)}+\bar{q}^{(n)} P\big(t,\int_0^L \bar{w}^{(n)}dx \big)-\bar{q}^{(n+1)} P\big(t,\int_0^L \bar{w}^{(n+1)}dx \big) \\
		{}  & + \Big(\gamma \bar{q}^{(n+1)}+\kappa\big(\bar{q}^{(n+1)}\big)^2-\gamma \bar{q}^{(n)}-\kappa\big(\bar{q}^{(n)}\big)^2\Big),
\end{align*}
and \cref{estim_q}, we get
$\|\delta u^{(n+1)}\| _{W^{1,2}_r(Q)}\leq C/n$ with the standard parabolic estimates. From Sobolev embedding \cref{Sobolev}, $\|\partial_x u^{(n+1)}-\partial_xu^{(n)}\| _{L^\infty(Q)}\leq C/n$. We then obtain
$$
\| q^{(n+1)}- q^{(n)}\|_{L^\infty(Q)}\leq \frac{\sup_{t}\left\vert \mathrm{P}^{(n)}(t)-\mathrm{P}^{(n-1)}(t)\right\vert+ \|\partial_x u^{(n)}-\partial_xu^{(n-1)}\| _{L^\infty(Q)}}{2\kappa} \leq  \frac{C}{n}
 $$   
 with \cref{q update 1}.                                                                                                                                                                                                       
\end{proof}
The following Lemma plays a pivotal role in the convergence analysis of Fictitious Play for mean field games. Its proof can be found in \cite[Lemma 2.7]{MR3608094}.
\begin{lemma}\label{a_n}
	Consider a sequence of positive real numbers $\{ a_n \}_{n \in \N}$ such that $\sum_{n=1}^{\infty} \frac{a_n}{n} < +\infty$. Then we have ${\lim_{N \to \infty} \frac{1}{N} \sum_{n=1}^{N} a_n = 0}$.
	In addition, if there is a constant $C >0 $ such that $\abs a_n - a_{n+1} \abs < C/n$, then $\lim_{n \to \infty} a_n = 0$.
\end{lemma}
Next, we state the convergence result.
\begin{theorem}\label{Maintheorem1} Under Assumption \ref{Main assumptions} and let $\zeta_n$ be defined by \cref{assumption-learning-rates}, the family $\{ (u^{(n)} , m^{(n)}, q^{(n)}) \}$, $n \in \N$, given by the algorithm (SPI) converges to the unique weak solution of the second order MFG system \eqref{MFG}.
\end{theorem}
\begin{proof}
\textit{Step 1}. Define
\begin{equation}\label{a_n 2}
	a_n=\int_Q m ^{(n+1)}|q^{(n+1)}-\bar{q}^{(n)}|^2dxdt,
\end{equation}
our aim in \textit{Step 1} is to show $a_n\rightarrow 0$. 
{We use the potential formulation from \cref{J}:
\begin{equation}\label{def J}
J(q,m)=\int_0^T e^{-\lambda t}\Phi(t,\int_0^L qmdx)dt-\int_Qe^{-\lambda t} \lc\gamma q+\kappa q^2\rc mdxdt.
\end{equation}
Recall the notation $\bar{w}^{(n)}=\bar{q}^{(n)}m^{(n)}$ from \cref{price update}. \ref{assumption3} implies $P(t,\cdot)$ is Lipschitz continuous w.r.t the second variable. Moreover, from \cref{C/n bound algo1} we have
$$
\left\vert \int_0^L \bar{w}^{(n+1)}dx-\int_0^L \bar{w}^{(n)}dx\right\vert \leq C\zeta_n.
$$
Then, there exists a constant $C>0$ such that
$$
\Phi(t,\int_0^L \bar{w}^{(n+1)}dx)-\Phi(t,\int_0^L \bar{w}^{(n)}dx)
\geq \mathrm{P}^{(n)}(t)\left(\int_0^L \bar{w}^{(n+1)}dx-\int_0^L \bar{w}^{(n)}dx\right)-C\zeta_n^2.
$$
We have
\begin{equation}\label{J1}
\begin{aligned}
&J(\bar{q}^{(n+1)},m^{(n+1)})-J(\bar{q}^{(n)},m^{(n)})\\
\geq & \int_0^T e^{-\lambda t}\mathrm{P}^{(n)}(t)\left(\int_0^L \bar{q}^{(n+1)}m^{(n+1)}dx-\int_0^L \bar{q}^{(n)}m^{(n)}dx\right)dt-C\zeta_n^2\\
&-\int_Q e^{-\lambda t}\left(\gamma \bar{q}^{(n+1)}+\kappa\big(\bar{q}^{(n+1)}\big)^2\right)m^{(n+1)}dxdt+\int_Q e^{-\lambda t}\left(\gamma \bar{q}^{(n)}+\kappa\big(\bar{q}^{(n)}\big)^2\right)m^{(n)}dxdt.
\end{aligned}
\end{equation}
We estimate the first term on the RHS of \cref{J1}. With \cref{qn+1-qn}, we have
\begin{equation}\label{J0}
\begin{aligned}
&\int_0^T e^{-\lambda t}\mathrm{P}^{(n)}(t)\left(\int_0^L \bar{q}^{(n+1)}m^{(n+1)}dx-\int_0^L \bar{q}^{(n)}m^{(n)}dx\right)dt\\
= {}& \zeta_n\int_Q e^{-\lambda t}\mathrm{P}^{(n)}(t)(q^{(n+1)}-\bar{q}^{(n)})m^{(n+1)}dxdt+\int_Q e^{-\lambda t}\mathrm{P}^{(n)}(t)\bar{q}^{(n)}(m^{(n+1)}-m^{(n)})dxdt.
\end{aligned}
\end{equation}
From \cref{u update 1},  we have
\begin{equation*}
\int_Q e^{-\lambda t} \mathrm{P}^{(n)}(t)\bar{q}^{(n)}(m^{(n+1)}-m^{(n)})dxdt=(\mathsf{A}_n)+(\mathsf{B}_n),
\end{equation*}
where
\begin{align*}
&(\mathsf{A}_n) :=\int_Q e^{-\lambda t}\left(-\partial_t u^{(n)}-\sigma^2(x) \partial^2_{xx}u^{(n)}+\lambda u^{(n)}+\bar{q}^{(n)} \partial_xu^{(n)}\right)(m^{(n+1)}-m^{(n)})dxdt,\\
&(\mathsf{B}_n) := \int_Q e^{-\lambda t}\left(\gamma \bar{q}^{(n)}+\kappa\big(\bar{q}^{(n)}\big)^2\right)(m^{(n+1)}-m^{(n)})dxdt.
\end{align*}
We have
\begin{equation*}
\begin{aligned}
(\mathsf{A}_n) 
={}& \int_Q e^{-\lambda t}\left(\partial_t u^{(n)}+\sigma^2(x) \partial^2_{xx}u^{(n)}-\lambda u^{(n)}-\bar{q}^{(n)} \partial_xu^{(n)}\right)m^{(n)}dxdt\\
& -\int_Q e^{-\lambda t}\left(\partial_t u^{(n)}+\sigma^2(x) \partial^2_{xx}u^{(n)}-\lambda u^{(n)}-\bar{q}^{(n+1)} \partial_xu^{(n)}\right)m^{(n+1)}dxdt\\
&+\zeta_n\int_Q e^{-\lambda t}(\bar{q}^{(n)}\partial_xu^{(n)}-q^{(n+1)}\partial_xu^{(n)})m^{(n+1)}dxdt.
\end{aligned}
\end{equation*}
From integration by parts (using the adjoint structure \cref{adjoint}),
$$
\begin{aligned}
&\int_Q e^{-\lambda t}\left(\partial_t u^{(n)}+\sigma^2(x) \partial^2_{xx}u^{(n)}-\lambda u^{(n)}-\bar{q}^{(n)} \partial_xu^{(n)}\right)m^{(n)}dxdt\\
={}&e^{-\lambda T}\int_0^Lm^{(n)}(T,x)u_T(x)dx-\int_0^Lu^{(n)}(0,x)m(0,x)dx,
\end{aligned}
$$
$$
\begin{aligned}
&\int_Q e^{-\lambda t}\left(\partial_t u^{(n)}+\sigma^2(x) \partial^2_{xx}u^{(n)}-\lambda u^{(n)}-\bar{q}^{(n+1)} \partial_xu^{(n)}\right)m^{(n+1)}dxdt\\
={}&e^{-\lambda T}\int_0^Lm^{(n+1)}(T,x)u_T(x)dx-\int_0^Lu^{(n)}(0,x)m(0,x)dx.
\end{aligned}
$$
Therefore, using again \cref{qn+1-qn}, we obtain
\begin{equation}\label{An}
(\mathsf{A}_n)=e^{-\lambda T}\int_0^L(m^{(n)}(T,x)-m^{(n+1)}(T,x))u_T(x)dx+\zeta_n\int_Q e^{-\lambda t}(\bar{q}^{(n)}\partial_xu^{(n)}-q^{(n+1)}\partial_xu^{(n)})m^{(n+1)}dxdt.
\end{equation}
Next, we consider
\begin{equation*}
\begin{aligned}
(\mathsf{B}_n) 
={}& \int_Q e^{-\lambda t}\left(\gamma \bar{q}^{(n)}+\kappa\big(\bar{q}^{(n)}\big)^2-\gamma \bar{q}^{(n+1)}-\kappa\big(\bar{q}^{(n+1)}\big)^2 \right)m^{(n+1)}dxdt\\
&+\int_Q e^{-\lambda t}\left(\gamma \bar{q}^{(n+1)}+\kappa\big(\bar{q}^{(n+1)}\big)^2\right)m^{(n+1)}dxdt-\int_Q e^{-\lambda t}\left(\gamma \bar{q}^{(n)}+\kappa\big(\bar{q}^{(n)}\big)^2\right)m^{(n)}dxdt.
\end{aligned}
\end{equation*}
Since
$
	\big(\bar{q}^{(n+1)}\big)^2\leq (1-\zeta_n) \big(\bar{q}^{(n)}\big)^2+\zeta_n (q^{(n+1)}\big)^2,
$
with \cref{qn+1-qn} we obtain
\begin{equation}\label{Bn}
\begin{aligned}
(\mathsf{B}_n) \geq {}& \zeta_n\int_Q e^{-\lambda t}\left(\gamma \bar{q}^{(n)}+\kappa\big(\bar{q}^{(n)}\big)^2-\gamma q^{(n+1)}-\kappa\big(q^{(n+1)}\big)^2\right)m^{(n+1)}dxdt\\
&+\int_Q e^{-\lambda t}\left(\gamma \bar{q}^{(n+1)}+\kappa\big(\bar{q}^{(n+1)}\big)^2\right)m^{(n+1)}dxdt-\int_Q e^{-\lambda t}\left(\gamma \bar{q}^{(n)}+\kappa\big(\bar{q}^{(n)}\big)^2\right)m^{(n)}dxdt.
\end{aligned}
\end{equation}
With \cref{J1}, \cref{J0}, results for $(\mathsf{A}_n)$ \cref{An} and $(\mathsf{B}_n)$ \cref{Bn} we obtain:
\begin{equation}\label{J2}
\begin{aligned}
&J(\bar{q}^{(n+1)},m^{(n+1)})-J(\bar{q}^{(n)},m^{(n)})\\
\geq {}& \zeta_n \int_Q e^{-\lambda t}\mathsf{F}_nm^{(n+1)}dxdt-C\zeta_n^2+e^{-\lambda T}\int_0^L(m^{(n)}(T,x)-m^{(n+1)}(T,x))u_T(x)dx,
\end{aligned}
\end{equation}
where 
\begin{equation}\label{F def}
\mathsf{F}_n(t,x):= \left(-q^{(n+1)} \left(\partial_xu^{(n)}-\mathrm{P}^{(n)}(t)+\gamma \right)-\kappa(q^{(n+1)})^2\right)                    -\left(-\bar{q}^{(n)} \left(\partial_xu^{(n)}-\mathrm{P}^{(n)}(t)+\gamma \right)-\kappa(\bar{q}^{(n)})^2\right).
\end{equation}
}
\par
We compare the term $\int_Q e^{-\lambda t}\mathsf{F}_nm^{(n+1)}dxdt$ with $a_n$. Recall from \cref{q update 1} that 
$$
-q^{(n+1)} \left(\partial_xu^{(n)}-\mathrm{P}^{(n)}(t)+\gamma \right)-\kappa(q^{(n+1)})^2=H(\partial_xu^{(n)}-\mathrm{P}^{(n)}(t)+\gamma),
$$
 we obtain
$
	\mathsf{F}_n \geq \kappa|q^{(n+1)}-\bar{q}^{(n)}|^2
$
by using \cref{H-L}. As $e^{-\lambda T}\leq e^{-\lambda t}$,
\begin{equation}\label{a_n bound}
	\frac{\beta\kappa e^{-\lambda T}}{n+\beta}a_n\leq  \int_Qe^{-\lambda t}\frac{\beta}{n+\beta}m^{(n+1)}\mathsf{F}_n(t,x)dxdt.
\end{equation}\par
Now we are ready to estimate $\sum_{n=1}^{\infty} a_n/n$. From \cref{J2} and \cref{a_n bound},
\begin{equation*}
\frac{\beta\kappa e^{-\lambda T}}{n+\beta}a_n
\leq  J(\bar{q}^{(n+1)},m^{(n+1)})-J(\bar{q}^{(n)},m^{(n)})+C\zeta_n^2+e^{-\lambda T}\int_0^L(m^{(n+1)}(T,x)-m^{(n)}(T,x))u_T(x)dx,
\end{equation*}
by telescoping sum, we obtain $\forall N\in \mathbb{N}$ and $N>1$
\begin{equation}\label{a_n telescope}
	\begin{aligned}
		\sum_{n=1}^{N}\frac{a_n}{n+\beta}
		\leq {}&\frac{ e^{\lambda T}}{\beta\kappa}\left( J(\bar{q}^{(N+1)},m^{(N+1)})-J(\bar{q}^{(1)},m^{(1)})\right) +\frac{C e^{\lambda T}}{\beta\kappa}\sum_{n=1}^{N}\zeta_n^2                                                                                    \\
		                                     & +\frac{1}{\beta\kappa}\int_0^L(m^{(N+1)}(T,x)-m^{(1)}(T,x))u_T(x)dx.
	\end{aligned}
\end{equation}
From \cref{Solution u m1}, we have that the right hand side of \cref{a_n telescope} is bounded for any $N>1$, therefore
$\sum_{n=1}^{\infty} a_n / n < +\infty$. From Lemma \ref{Solution u m1} and Lemma \ref{C/n bound algo1}, we have that $\abs a_n - a_{n+1} \abs < C/n$. Then we obtain $a_n\rightarrow 0$ from Lemma \ref{a_n}.\par
\textit{Step 2}. 
{
For each $n$, $(u^{(n)},m^{(n)},q^{(n)})\in \mathbbm{X}$ and $\bar{q}^{(n)}\in L^\infty(Q)$. Since the domain $Q$ is bounded, $(\bar{q}^{(n)})^2\in L^{r}(Q)$. From \cref{Solution u m1} and Sobolev embedding \cref{Sobolev}, we can extract a subsequence $(u^{(n_i)},m^{(n_i)},q^{(n_i)},\bar{q}^{(n_i)})$ such that 
\begin{center}
 $u^{(n_i)}$ converges weakly in $W^{1,2}_r(Q)$,\\
 $\partial_xu^{(n_i)}$ and $m^{(n_i)}$ converge uniformly, \\
 $q^{(n_i)}$ and $\bar{q}^{(n_i)}$ converge in $L^\infty$-weak$^*$,\\
 $q^{(n_i)}(t,\cdot)$ and $\bar{q}^{(n_i)}(t,\cdot)$ converge in $L^\infty$-weak$^*$ a.e. in $[0,T]$,\\
 $(\bar{q}^{(n_i)})^2$ converges weakly in $L^r(Q)$.
 \end{center}
We denote the limit by $(\hat{u},\hat{m},\hat{q},\check{q})$. $L^\infty$-weak$^*$ convergence of $\bar{q}^{(n_i)}$ to $\check{q}$ and uniform convergence of $m^{(n_i)}$ to $\hat{m}$ imply $L^\infty$-weak$^*$ convergence of $\bar{q}^{(n_i)}m^{(n_i)}$ to $\check{q}\hat{m}$. This implies $\text{div} (m^{(n)} \bar{q}^{(n)})$ converges to $\text{div} (\hat{m} \check{q})$ in $L^2(0,T;H^{-1}(\mathbb{L}))$. We can pass to the limit in \cref{m update 1} and $\hat{m}$ is a bounded weak solution (\cref{def weak sol}) to 
$$
\begin{aligned}
  \qquad & \partial_t \hat{m}-\partial^2_{xx}\big(\sigma^2(x)  \hat{m}\big)-\text{div} (\hat{m} \check{q})=0,        \\
  & \hat{m}(t,0)=0,\, \partial_x\big(\sigma^2(L) \hat{m}(t,L)\big)+\check{q}(t,L)\hat{m}(t,L)=0,  \,\hat{m}(0,x)=m_0(x).
\end{aligned}
$$
Similarly, $\int_0^L\bar{q}^{(n_i)}(t,x)m^{(n_i)}(t,x)dx$ converges to $\int_0^L\check{q}(t,x)\hat{m}(t,x)dx$ a.e. in $[0,T]$, with \ref{assumption3} we have $\mathrm{P}^{(n)}(t)$ converges to $P\big(t,\int_0^L\check{q}(t,x)\hat{m}(t,x)dx\big)$ a.e. in $[0,T]$. $\mathrm{P}^{(n)}$ is uniformly bounded in $L^\infty(0,T)$, hence from Egorov theorem (cfr. \cite[Theorem 4.19]{brezis2011functional}, p.115 and p.123) $\mathrm{P}^{(n)}$ converges strongly to $P\big(\cdot,\int_0^L\check{q}\hat{m}dx\big)$ in $L^r(0,T)$. Again with $L^\infty$-weak$^*$ convergence of $\bar{q}^{(n_i)}$, $\bar{q}^{(n_i)}\mathrm{P}^{(n)}$ converges to $\check{q}P\big(\cdot,\int_0^L\check{q}\hat{m}dx\big)$ weakly in $L^r(Q)$. Moreover, with the uniform convergence of $\partial_xu^{(n_i)}$, $\bar{q}^{(n_i)}\partial_xu^{(n_i)}$ converges weakly to $\check{q}\partial_x\hat{u}$ in $L^r(Q)$. We can pass to the limit in \cref{u update 1} and \cref{q update 1} to obtain
$$
\begin{aligned}
			       \qquad & \partial_t \hat{u}+\sigma^2(x) \partial^2_{xx}\hat{u}-\lambda \hat{u}-\check{q}\partial_x\hat{u}+\check{q}P\big(t,\int_0^L\check{q}\hat{m}dx\big)-\gamma \check{q}-\kappa\big(\check{q}\big)^2=0,        \\
			       &\hat{u}(t,0)=0,\,\partial_x\hat{u}(t,L)=0, \,\hat{u}(T,x)=u_T(x),\\
& \hat{q}(t,x)= \lc \frac{P\big(t,\int_0^L\check{q}(t,x)\hat{m}(t,x)dx\big)-\gamma-\partial_x\hat{u}(t,x)}{2\kappa} \rc_+,
\end{aligned}
$$
where the last equation is satisfied a.e. in $Q$. Since $m ^{(n+1)}|q^{(n+1)}-\bar{q}^{(n)}|^2$ is uniformly bounded, from Lebesgue dominated convergence theorem and $a_n\rightarrow 0$ we have 
$\int_Q\hat{m}\vert \hat{q}-\check{q}\vert^2dx dt=0,$
 therefore 
$\hat{m}\vert \hat{q}-\check{q}\vert^2=0$ a.e. in $Q$. It follows from \cref{m stability} that $\hat{m}>0$ in $Q$, hence we have $\hat{q}=\check{q}$ a.e. in $Q$. Therefore, $(\hat{u},\hat{m},\hat{q})$ is a solution to system~\eqref{MFG}. From \cref{uniqueness}, the solution to the system \eqref{MFG} is unique in $\mathbbm{X}$, hence all subsequences of $(u^{(n)},m^{(n)},q^{(n)})$ converge to the same limit. We conclude that the whole sequence $(u^{(n)},m^{(n)},q^{(n)})$ converges to the solution of system \eqref{MFG}.
}
\end{proof}

{Next we show that if  $(u,m,{q^*})\in \mathbbm{X}$ is a solution to system~\eqref{MFG}, then $u$  is a classical solution. This is rather straight forward for some mean field games (cfr. \cite{tang2023learning}), but much more delicate for mean field games of controls. Similar analysis has been considered in \cite{bonnans2021schauder}, but the constraint $q\geq 0$ introduced additional difficulty. The key to overcome this difficulty is to use the monotonicity of $P(t,a)$. We first show that $P(t,\psi(t))$ is H\"older continuous in time. }
\begin{lemma}\label{etaHolder}
Under  Assumption \ref{Main assumptions} and suppose $(u,m,{q^*})$ in $\mathbbm{X}$ is the weak solution to system~\eqref{MFG}. Then $\|\psi(\cdot)\|_{C^{\alpha/2}(0,T)}\leq C$, where $\psi(t)$ is defined as in \cref{def w}. Moreover, $\mathrm{P}(\cdot)\in C^{\alpha/2}(0,T)$.
\end{lemma}
\begin{proof}
Let $0\leq t_1\leq T$, $0\leq t_2\leq T$ and $t_1\neq t_2$. Using \cref{def w} we denote ${q^*}(t_1,x) =(f_1)_+$ and ${q^*}(t_2,x)=(f_2)_+$ with
	$f_1=\frac{\mathrm{P}( t_1)-\gamma-\partial_xu(t_1,x)}{2\kappa}$, 
	$f_2=\frac{\mathrm{P}( t_2)-\gamma-\partial_xu(t_2,x)}{2\kappa}$.\par
First, we consider the case $\psi(t_1)\geq \psi(t_2)$. Recall $(f)_+=\frac{f+|f|}{2}$,
\begin{equation}\label{psi1-psi2}
	\begin{aligned}
		&\psi(t_1)-\psi(t_2)\\
		={} & \frac{1}{2}\int_0^L(f_1m(t_1,x)-f_2m(t_2,x))dx+\frac{1}{2}\int_0^L(|f_1|m(t_1,x)-|f_2|m(t_2,x))dx \\
		={}                    & \frac{1}{2}\int_0^L(f_1-f_2)m(t_1,x)dx+\frac{1}{2}\int_0^Lf_2\lc m(t_1,x)-m(t_2,x)\rc dx        \\
		                       & +\frac{1}{2}\int_0^L(|f_1|-|f_2|)m(t_1,x)dx+\frac{1}{2}\int_0^L|f_2|\big(m(t_1,x)-m(t_2,x)\big)dx \\
		\leq{}                 & \frac{1}{2}\int_0^L\big(f_1-f_2+|f_1-f_2|\big)m(t_1,x)dx+ \int_0^L|f_2||m(t_1,x)-m(t_2,x)|dx.
	\end{aligned}
\end{equation}
Since $\psi(t_1)-\psi(t_2)\geq 0$, $P\big( t_2,\psi(t_1)\big)\leq P\big( t_2,\psi(t_2)\big)$ from Assumption~\ref{assumption3}. Hence
$
	P\lc  t_2,\psi(t_1)\rc-P\lc t_2,\psi(t_2)\rc+\lv P\lc t_2,\psi(t_1)\rc-P\lc t_2,\psi(t_2)\rc \rv=0.
$
We observe that
\begin{equation*}
	\begin{aligned}
		  & f_1-f_2+|f_1-f_2|    \\                                                                                                                                              
		= & {}\frac{1}{4\kappa}\lc P\lc t_1,\psi(t_1)\rc-P\lc t_2,\psi(t_2)\rc \rc-\frac{1}{4\kappa}\big(\partial_xu(t_1,x)-\partial_xu(t_2,x)\big)                       \\
		  & +\left\vert \frac{1}{4\kappa}\Big(P\big( t_1,\psi(t_1)\big)-P\big( t_2,\psi(t_2)\big)\Big)-\frac{1}{4\kappa}\big(\partial_xu(t_1,x)-\partial_xu(t_2,x)\big)\right\vert,
	\end{aligned}
\end{equation*}
\begin{equation*}
\frac{1}{4\kappa}\Big(P\big( t_1,\psi(t_1)\big)-P\big( t_2,\psi(t_2)\big)\Big)
\leq \frac{1}{4\kappa}P\big( t_2,\psi(t_1)\big)-P\big( t_2,\psi(t_2)\big)+\frac{1}{4\kappa}\left\vert P\big( t_1,\psi(t_1)\big)-P\big( t_2,\psi(t_1)\big)\right\vert
\end{equation*}
and
\begin{equation*}
	\begin{aligned}
		        & \left\vert \frac{1}{4\kappa}\Big(P\big( t_1,\psi(t_1)\big)-P\big( t_2,\psi(t_2)\big)\Big)-\frac{1}{4\kappa}\big(\partial_xu(t_1,x)-\partial_xu(t_2,x)\big)\right\vert                                                                                 \\
		\leq {} & \frac{1}{4\kappa} \lv P\big( t_2,\psi(t_1)\big)-P\big( t_2,\psi(t_2)\big)\rv+\frac{1}{4\kappa}\left\vert P\big( t_1,\psi(t_1)\big)-P\big( t_2,\psi(t_1)\big)\right\vert +\frac{1}{4\kappa}\left\vert \partial_xu(t_1,x)-\partial_xu(t_2,x)\right\vert.
	\end{aligned}
\end{equation*}
Therefore
\begin{equation*}
	\begin{aligned}
		  f_1-f_2+|f_1-f_2|                                                                                                                               
		\leq {}& \frac{1}{4\kappa}\Big(P\big( t_2,\psi(t_1)\big)-P\big( t_2,\psi(t_2)\big)+\lv P\big( t_2,\psi(t_1)\big)-P\big( t_2,\psi(t_2)\big)\rv \Big)                               \\
		     & +\frac{1}{2\kappa}\left\vert P\big( t_1,\psi(t_1)\big)-P\big( t_2,\psi(t_1)\big)\right\vert+\frac{1}{2\kappa}\left\vert \partial_xu(t_1,x)-\partial_xu(t_2,x)\right\vert \\
		\leq {}& \frac{1}{2\kappa}\left\vert P\big( t_1,\psi(t_1)\big)-P\big( t_2,\psi(t_1)\big)\right\vert+\frac{1}{2\kappa}\left\vert \partial_xu(t_1,x)-\partial_xu(t_2,x)\right\vert.
	\end{aligned}
\end{equation*}
$\|u\| _{W^{1,2}_r(Q)}\leq C$ implies $\|\partial_xu\| _{\mathcal C^{\alpha/2,\alpha}(Q)}\leq C$ by \cref{Sobolev}. From $f_2\in L^\infty(Q)$ and $m\in \mathcal{C}^{\alpha/2,\alpha}(Q)$, cfr. \cref{m stability}, it follows $\int_0^L|f_2||m(t_1,x)-m(t_2,x)|dx\leq C|t_1-t_2|^{\alpha/2}$. Therefore, from \cref{psi1-psi2} we obtain
$0\leq \psi(t_1)-\psi(t_2)\leq C|t_1-t_2|^{\alpha/2}$.\par
Second, we assume $\psi(t_1)-\psi(t_2)\leq 0$. Then we can follow the same reasoning to obtain $\psi(t_2)-\psi(t_1)\leq C|t_1-t_2|^{\alpha/2}$. From \ref{assumption3}, the price $P\lc t,\psi(t)\rc$ is $\mathcal{C}^{1}(0,T)$ for the $t$-variable. Since
\begin{equation*}
	\lv \mathrm{P}(t_1)-\mathrm{P}(t_2)\rv \leq \sup_{t_1\leq t\leq t_2}\lv \partial_tP\lc t,\psi(t) \rc\rv \lv t_1-t_2\rv
	+\sup_{t_1\leq t\leq t_2}\lv \partial_{\psi}P(t,\psi(t))\rv \lv \psi(t_1)-\psi(t_2)\rv,
\end{equation*}
we can conclude $\mathrm{P}(\cdot)\in C^{\alpha/2}(0,T)$.
\end{proof}

\begin{theorem}\label{weak-classical}
	Suppose $(u,m,{q^*})\in \mathbbm{X}$ is a solution to system~\eqref{MFG}, then
	$$
		(u,m,{q^*})\in \mathcal C^{1+\alpha/2,2+\alpha}(Q)\times \mathcal C^{(1+\alpha)/2,1+\alpha}(Q)\times \mathcal C^{\alpha/2,\alpha}(Q).
	$$
\end{theorem}
\begin{proof}
It follows from \cref{Maintheorem1} that $\partial_xu\in \mathcal C^{\alpha/2,\alpha}(Q)$. From \cref{etaHolder}  and \cref{controlrecursive} we then obtain ${q^*}\in \mathcal C^{\alpha/2,\alpha}(Q)$. We have $m\in \mathcal C^{(1+\alpha)/2,1+\alpha}(Q)$. By Schauder estimate, $u\in \mathcal C^{1+\alpha/2,2+\alpha}(Q)$.
\end{proof}
{A useful measure of convergence of the algorithm is \textit{exploitability}. It has been discussed in \cite{perrin2020fictitious} for MFG reinforcement learning and in \cite{MR4659381} for considering the convergence of conditional gradient method for solving a MFG of controls. To define it, consider the solution to the HJB equation, where $\mathrm{P}^{(n)}(t)$ is given and the unknown is $v^{(n)}$:
\begin{equation}\label{best response}
		      \left\{\begin{aligned}
			       \quad &\partial_t v^{(n)} +\sigma^2(x) \partial^2_{xx}v^{(n)} -\lambda v^{(n)} +\sup_{q\geq 0}\left\{-q\partial_xv^{(n)}+q\mathrm{P}^{(n)}(t)-\gamma q-\kappa q^2\right\}=0, \\ 
			       \quad &v^{(n)}(t,0)=0,\,\partial_xv^{(n)}(t,L)=0, \,v^{(n)}(T,x)=u_T(x).
		      \end{aligned}\right.
	      \end{equation}
	      We define 
\begin{equation}\label{def exploit}	     
	     \Gamma_n=\int_0^Lv^{(n)}(0,x)m_0(x)dx-\int_0^Lu^{(n)}(0,x)m_0(x)dx
\end{equation}
as the \textit{exploitability} at iteration $n$. It quantifies the \textit{average} gain of a representative agent to replace its policy by the best response while the rest of the population stays with the strategy $\bar{q}^{(n)}$. From  $v^{(n)}(t,0)-u^{(n)}(t,0)=0$, $\partial_x(v^{(n)}-u^{(n)}) (t,L)=0$, $v^{(n)}(T,x)-u^{(n)}(T,x)=0$ and 
	      $$ 
	      \begin{aligned}
	     &- \partial_t (v^{(n)}-u^{(n)}) -\sigma^2(x) \partial^2_{xx}(v^{(n)}-u^{(n)}) +\lambda (v^{(n)}-u^{(n)})\\
	     ={}&\sup_{q\geq 0}\left\{-q\partial_xv^{(n)}+q\mathrm{P}^{(n)}(t)-\gamma q-\kappa q^2\right\}-\lc-\bar{q}^{(n)} \partial_xu^{(n)}+\bar{q}^{(n)} \mathrm{P}^{(n)}(t)-\gamma \bar{q}^{(n)}-\kappa\lc\bar{q}^{(n)}\rc^2\rc \geq 0,
	      \end{aligned}
	      $$
	      one can obtain from weak maximum principle that $v^{(n)}(t,x)\geq u^{(n)}(t,x)$. With \ref{assumption1}, we have $\Gamma_n\geq 0$. Moreover, $\Gamma_n= 0$ implies $v^{(n)}(0,x)= u^{(n)}(0,x)$ a.e. in $\mathbb{L}$, hence $\Gamma_n= 0$ only if the Nash equilibrium has been reached. 
}	      
	      \section{Numerical analysis}
In order to solve the Cournot MFG of controls system numerically, we use a finite difference discretized form of the learning algorithm (SPI) defined in \cref{sec:learning}. We recall that  finite difference methods for solving MFG systems has been developed   in \cite{MR2679575,MR2888257} and  have been used in \cite{MR4146720} for solving MFG of controls.
\subsection{Numerical method}
We fix a step size  $h=L/N_L$ in space and   $\Delta t=T/N_T$  in time. We discretize   $Q=[0,T]\times \mathbb{L}$ with   a grid  $ (t_\tau,x_i) \in\mathcal{Q}_{h,\Delta t}$ where $t_\tau=\tau \Delta t$, $\tau \in \{0,...,N_T\}$, and $x_i=ih$, $i\in \{0,...,N_L+1\}$. In order to approximate the boundary condition at $x=L$, we add a ghost point $x_{N_L+1}=(N_L+1)h$. We approximate functions on $Q$ by vectors $(\varphi_{\tau,i})_{\tau \in \{0,...,N_T\}, i\in \{0,...,N_L\}}  \in \mathbb{R}^{(N_T+1)(N_L+2)}$, and $\varphi_{\tau,i}$ approximates the value of a function $\varphi$ at $(t_\tau,x_i)$. We denote by $\varphi_\tau$ the vector $(\varphi_{\tau,i})_{i\in \{0,...,N_L+1\}} \in \mathbb{R}^{N_L+2}$. We then define the following discrete second and first order differential operators:
\begin{equation}\label{discrete Laplacian}
	(\Delta_\sharp \varphi_{\tau})_i =\frac{1}{h^2}\left(\varphi_{\tau,i-1}-2\varphi_{\tau,i}+\varphi_{\tau,i+1}\right), (D_\sharp \varphi_{\tau})_i =\frac{1}{h}\left( \varphi_{\tau,i}-\varphi_{\tau,i-1}\right),\quad 1 \le i \le N_L,
\end{equation}
and the divergence
\begin{equation}\label{discretediv}
	\begin{aligned}
		\text{div}_\sharp(\varphi_{\tau}\,\mathrm{Q}_\tau)_i     & =\frac{1}{h}
		\left( \varphi_{\tau,i+1}\mathrm{Q}_{\tau,i+1} - \varphi_{\tau,i} \mathrm{Q}_{\tau,i}\right),\,\,\,0\leq i\leq N_L-1, \\
		\text{div}_\sharp(\varphi_{\tau}\,\mathrm{Q}_\tau)_{N_L} & =\frac{1}{h}
		\left( - \varphi_{\tau,N_L} \mathrm{Q}_{\tau,N_L}\right).
	\end{aligned}
\end{equation}
Note that, following   the ideas developed in \cite[pp. 67-68]{MR3601001} for MFG on networks, the operator $\text{div}_\sharp$ is defined differently for $0\leq i\leq N_L-1$ and $i=N_L$ in order to deal with the flux boundary condition at the point $x=L$.
We define the discrete $\ell^2$ norm of $\varphi  \in \mathbb{R}^{(N_T+1)(N_L+2)}$ by:
$
	\|\varphi \|_{\ell^2(\mathcal{Q}_{h,\Delta t})}=\lc\sum_{\tau,i}|\varphi_{\tau,i}|^2h\Delta t\rc^{1/2}
$.
We now present the discretized forms of the equations \cref{m update 1,price update,u update 1,q update 1,average} that need to be solved during iteration $n$ of the (SPI) algorithm. The update steps \cref{m update 1}-\cref{average} become:
	\begin{equation}\label{M update 1}
		      \begin{aligned}
			       & \frac{M^{(n)}_{\tau+1,i}-M^{(n)}_{\tau,i}}{\Delta t}-(\Delta_\sharp \sigma^2 M^{(n)}_{\tau+1})_i-\text{div}_\sharp(M^{(n)}_{\tau+1}\,\bar{\mathrm{Q}}^{(n)}_\tau)_i=0\,\,\forall i\in \{0,...,N_L\}, \\
			       & M^{(n)}_{\tau,0}=0,\,  \sigma^2_{N_L}M^{(n)}_{\tau,N_L}=\sigma^2_{N_L+1}M^{(n)}_{\tau,N_L+1}.
		      \end{aligned}
	      \end{equation}
 \begin{equation}\label{disc P update}
		      \mathrm{P}^{(n)}_{\tau}=P\Big(\tau,\sum_i\bar{W}^{(n)}_{\tau,i}h\Big),\,\,\bar{W}^{(n)}_{\tau,i}=M^{(n)}_{\tau+1,i}\,\bar{\mathrm{Q}}^{(n)}_{\tau,i}.
	      \end{equation}
\begin{equation}\label{U update 1}
		      \begin{aligned}
			       & \frac{U_{\tau+1,i}^{(n)}-U_{\tau,i}^{(n)}}{\Delta t}+\sigma_i^2(\Delta_\sharp U^{(n)}_{\tau})_i-\lambda U_{\tau,i}^{(n)}+\bar{\mathrm{Q}}^{(n)}_{\tau,i}\Big(\mathrm{P}^{(n)}_{\tau}-(D_\sharp U^{(n)}_{\tau})_i\Big)
			       -\gamma \bar{\mathrm{Q}}^{(n)}_{\tau,i}-\kappa \big(\bar{\mathrm{Q}}^{(n)}_{\tau,i}\big)^2=0
			       \quad \forall i\in \{0,...,N_L\},  \\                                                                                                                                                                                                                                                                                        &\quad U^{(n)}_{\tau,0}=0,\,U^{(n)}_{\tau,N_L}=U^{(n)}_{\tau,N_L+1}.
		      \end{aligned}
	      \end{equation}
\begin{align}\label{Q update}
		      \mathrm{Q}^{(n+1)}_{\tau,i}&=\min \left\{\Big(\frac{\mathrm{P}^{(n)}_{\tau}-\gamma-(D_\sharp U^{(n)}_{\tau})_i}{2\kappa}\Big)_+,\frac{C_P}{2\kappa}\right\},\\
\label{disc average}
		      \bar{\mathrm{Q}}^{(n+1)}_{\tau,i}&=(1-\zeta_n)\bar{\mathrm{Q}}^{(n)}_{\tau,i}+\zeta_n{\mathrm{Q}}^{(n+1)}_{\tau,i}.
	      \end{align}
	      {For each $\mathrm{P}^{(n)}_{\tau}$, we can use the following scheme to discretize the HJB \cref{best response}, where $v^{(n)}(t,x)$ is approximated by $V^{(n)}_{\tau,i}$.
\begin{equation}\label{V}
		      \begin{aligned}
			       & \frac{V_{\tau+1,i}^{(n)}-V_{\tau,i}^{(n)}}{\Delta t}+\sigma_i^2(\Delta_\sharp V^{(n)}_{\tau})_i-\lambda V_{\tau,i}^{(n)}+\sup_{\mathrm{Q}_{\tau,i}}\Big\{\mathrm{Q}_{\tau,i}\Big(\mathrm{P}^{(n)}_{\tau}-(D_\sharp V^{(n)}_{\tau})_i\Big)
			    -\gamma \mathrm{Q}_{\tau,i}-\kappa \big(\mathrm{Q}_{\tau,i}\big)^2\Big\}=0\\
			     & \forall i\in \{0,...,N_L\},                                                                                                                                                                                                                                                                                          \quad V^{(n)}_{\tau,0}=0,\,V^{(n)}_{\tau,N_L}=V^{(n)}_{\tau,N_L+1}.
		      \end{aligned}
		 \end{equation}
		 }
We introduce the following \textit{discrete smoothed policy iteration algorithm}. 
\begin{algorithm}\label{alg:SPI-Cournot-MFG}
\caption{Discretized Smoothed Policy Iteration for Cournot MFG}
\KwData{Initial $\mathrm{Q}^{(0)}_{\tau,i}$ such that $0\leq \mathrm{Q}^{(0)}_{\tau,i} \leq \frac{C_P}{2\kappa}$, for all $\tau, i$; Set $n = 0$ and $\mathrm{Q}^{(0)} = \bar{\mathrm{Q}}^{(0)}$.}
\KwResult{solution $(U,M,\mathrm{Q})$.}
\Do{$\|\mathrm{Q}^{(n+1)}-\bar{\mathrm{Q}}^{(n)}\|_{\ell^2(\mathcal{Q}_{h,\Delta t})}> \epsilon$}{
\textit{Mean field update. } Set $M^{(n)}_{0}=M_{0}$ and solve \cref{M update 1} for all $\tau\in \{0,\dots, N_T-1\}$.\\
\textit{Price update. } Let $\mathrm{P}_\tau^{(n)}$ be defined by~\cref{disc P update}  for all $\tau\in \{0,\dots, N_T-1\}$.\\
\textit{Policy evaluation. } Set $U^{(n)}_{N_T}=U_{N_T}$ and solve \cref{U update 1} for all $\tau\in \{0,\dots, N_T-1\}$.\\
\textit{Policy update. } Let $\mathrm{Q}_\tau^{(n+1)}$ be defined by~\cref{Q update} for all $\tau\in \{0,\dots, N_T-1\}$.\\
\textit{Policy smoothing. } Let  $\bar{\mathrm{Q}}^{(n+1)}_{\tau}$ be defined by~\cref{disc average} for all $\tau\in \{0,\dots, N_T-1\}$.\\
Increment $n$.
}
\end{algorithm}

{Notice that in algorithm $1$, $V^{(n)}$ is not needed if we only want to compute the solution to system \cref{MFG}. To illustrate the numerical convergence of exploitability $\Gamma_n$ from \cref{def exploit} for conceptual analysis, we can use a variant of algorithm $1$ where we solve \cref{V} for  $V^{(n)}$, between updating $\mathrm{P}_\tau^{(n)}$ and updating $\mathrm{Q}_\tau^{(n+1)}$. For solving \cref{V}, we use an inner iteration loop (indexed by $\mathrm{k}$) between \textit{Policy evaluation} and \textit{Policy update} until $V^{(n,\mathsf{k})}$ and $V^{(n,\mathrm{k}+1)}$ are sufficiently close. For numerical approximation of $\Gamma_n$ in \cref{def exploit} we use $\sum_ihV_{0,i}^{(n)}M_0-\sum_ihU_{0,i}^{(n)}M_0$.}

\begin{remark}
	In particular, the discretized second order derivatives in \cref{M update 1,U update 1} are
	\begin{align*}
		(\Delta_\sharp \sigma^2 M^{(n)}_{\tau+1})_i & =\frac{1}{h^2}\lc \sigma^2_{i-1}M^{(n)}_{\tau+1,i-1}-2\sigma^2_iM^{(n)}_{\tau+1,i}+\sigma^2_{i+1}M^{(n)}_{\tau+1,i+1}\rc, \\
		\sigma_i^2(\Delta_\sharp U^{(n)}_{\tau})_i  & =\frac{\sigma^2_i}{h^2}\lc U^{(n)}_{\tau,i-1}-2U^{(n)}_{\tau,i}+U^{(n)}_{\tau,i+1}\rc.
	\end{align*}
	We denote
	$
		\mathcal A:=(\sigma^2\Delta_\sharp \cdot)_i-\mathrm{Q}_{\tau,i}(D_\sharp \cdot)_i$ as an linear operator acting on $\varphi$ such that $\varphi_{\tau,0}=0$, $\varphi_{\tau,N_L}=\varphi_{\tau,N_L+1}$. Analogously, denote by $\mathcal A^*:=(\Delta_\sharp \sigma^2 \cdot)_i+\text{div}_\sharp(\mathrm{Q}_\tau \cdot)_i$ an operator acting on $\varphi$ such that 
$\varphi_{\tau,0}=0$, $\sigma^2_{N_L}\varphi_{\tau,N_L}=\sigma^2_{N_L+1}\varphi_{\tau,N_L+1}$. From summation by parts, $\mathcal A$ and $\mathcal A^*$ are adjoint operators, i.e. \cref{adjoint} is preserved at the discrete level. On the theoretical side, this implies our numerical scheme is justified even when the solution $m$ to the FPK is in the weak sense. On the practical side, we may construct a sparse matrix for $\mathcal A$ and obtain $\mathcal A^*$ directly by transposing it.
\end{remark}
\begin{remark}
	We show the consistency of the scheme with the flux boundary condition for $m$ at $x=L$, if $\frac{\Delta t}{h}$ is kept constant. Note that the scheme, being implicit, does not need to satisfy the CFL condition. Since $\sigma^2_{N_L}M_{\tau,N_L}=\sigma^2_{N_L+1}M_{\tau,N_L+1}$,  then
	$$
		-(\Delta_\sharp \sigma^2 M_{\tau+1})_{N_L}=\frac{\sigma^2_{N_L} M_{\tau+1,N_L}-\sigma^2_{N_L-1}M_{\tau+1,N_L-1}}{h^2},
	$$
	and therefore the first equation in \cref{M update 1} can be written as
	$$
		h\frac{M_{\tau+1,N_L}-M_{\tau,N_L}}{\Delta t}+\frac{\sigma^2_{N_L} M_{\tau+1,N_L}-\sigma^2_{N_L-1}M_{\tau+1,N_L-1}}{h}+
		M_{\tau+1,N_L} \mathrm{Q}_{\tau,N_L}=0.
	$$
	We let $\Delta t, h\rightarrow 0$ with $\Delta t/h$ constant to obtain the boundary condition for $m$ in  \eqref{MFG} (\textit{ii}).
\end{remark}

\subsection{Numerical results}

In the first example, we consider the domain with $L=6$ and $T=15$. The price function is defined by \cref{DP2} . For time and space discretization, we use $N_x=300$ and $N_t=2000$. The initial condition is 
$$
\displaystyle m_0(x)=\frac{\lc e^{-0.2(x-3)^2}-0.7\rc_+}{\int_0^L\lc e^{-0.2(x-3)^2}-0.7\rc_+ dx}.
$$
For the first example, we consider both cases with the diffusion process being either a Brownian motion (BM) or a geometric Brownian motion (GBM). We set the parameters as $\lambda=0,\,\rho=0.01,\,\sigma=0.1,\,\eta=1.2,\,\delta=0.2,\,\gamma=2,\,\kappa=5,\, E=3.$
$$
\text{Test 1 (BM)}: \sigma^2(x)=\sigma^2,\,\text{Test 1 (GBM)}: \sigma^2(x)=(\sigma x)^2.
$$
\begin{figure}[h!]
	\centering
	\caption{Test 1 (BM): convergence}\label{an}
	\begin{tabular}{cc}
	\includegraphics[width=0.45\textwidth]{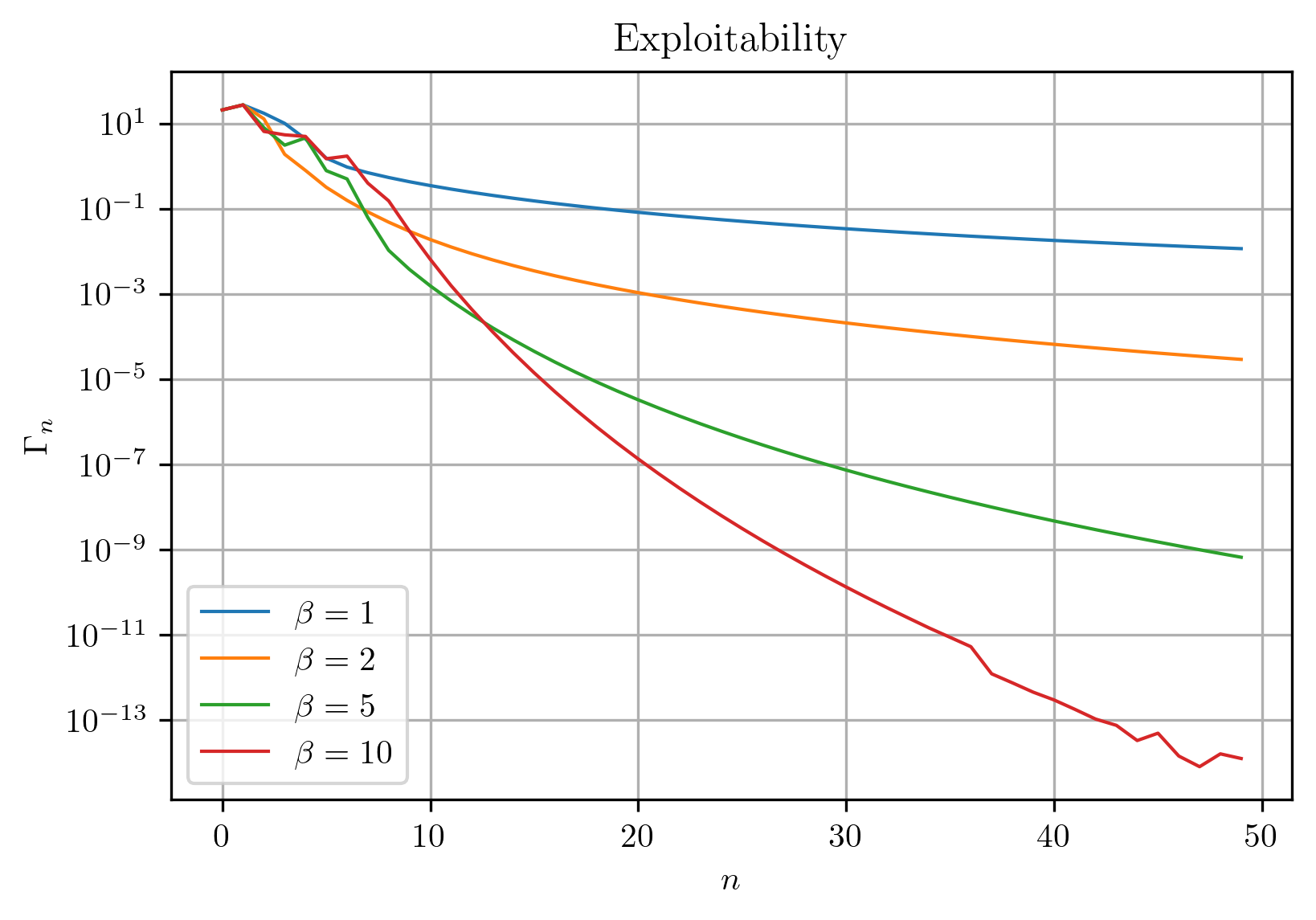} &
	\includegraphics[width=0.45\textwidth]{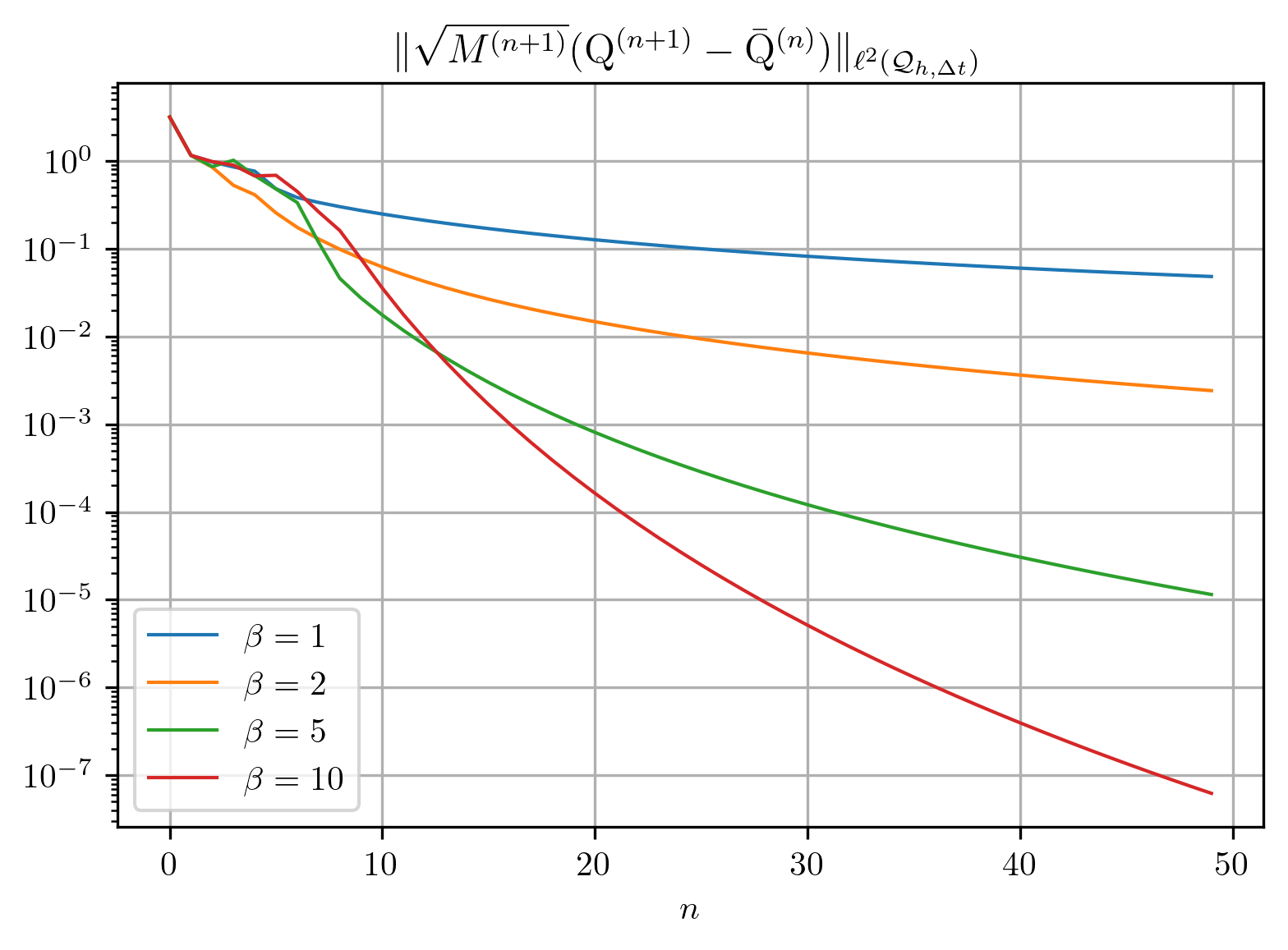} \\
	\end{tabular}
\end{figure}

\begin{figure}[h!]
	\centering
	\caption{Solution to test 1 (BM): density (left) and value function (right)}\label{test1 sol1}
	\begin{tabular}{cc}
	\includegraphics[width=0.45\textwidth]{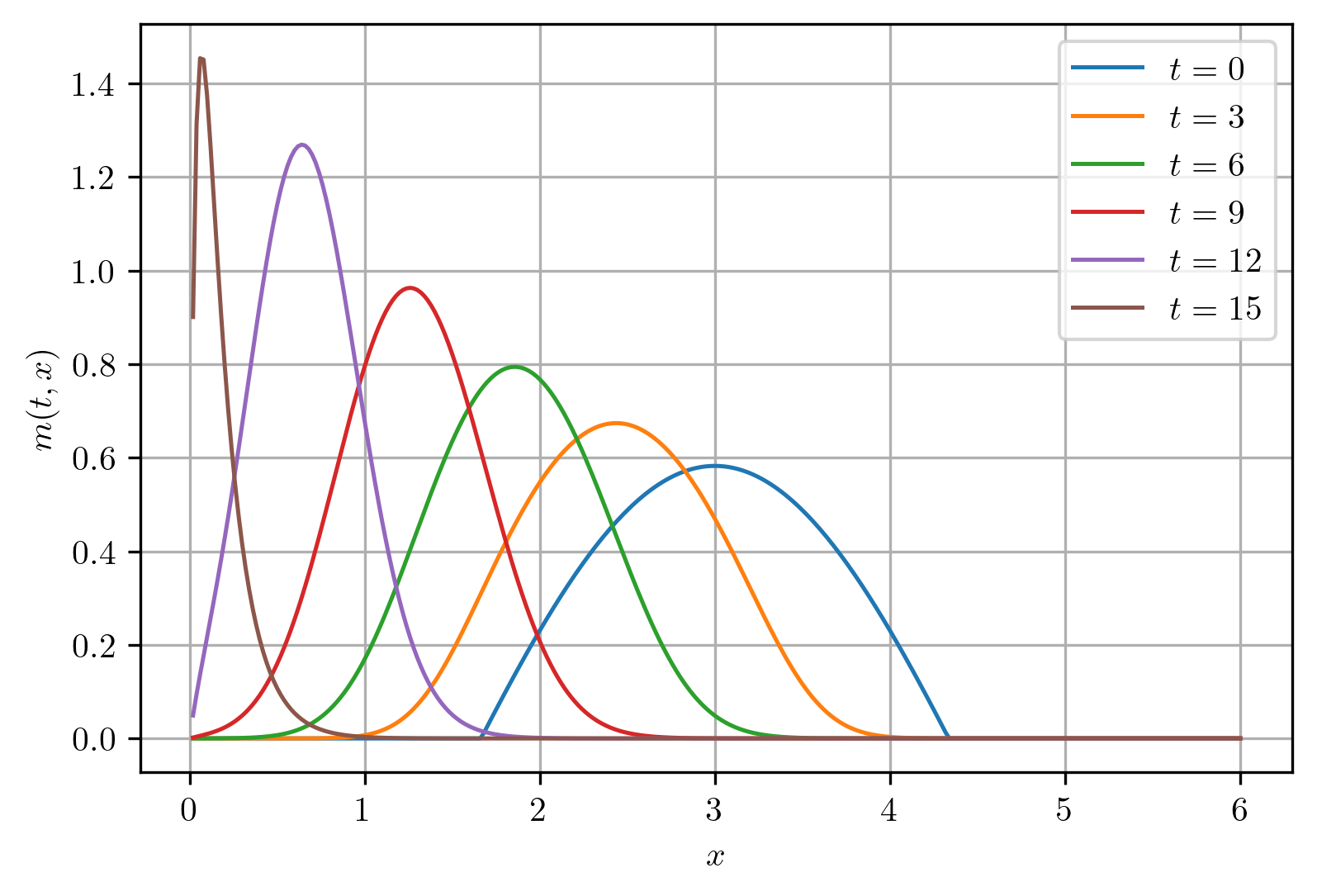} &
	\includegraphics[width=0.45\textwidth]{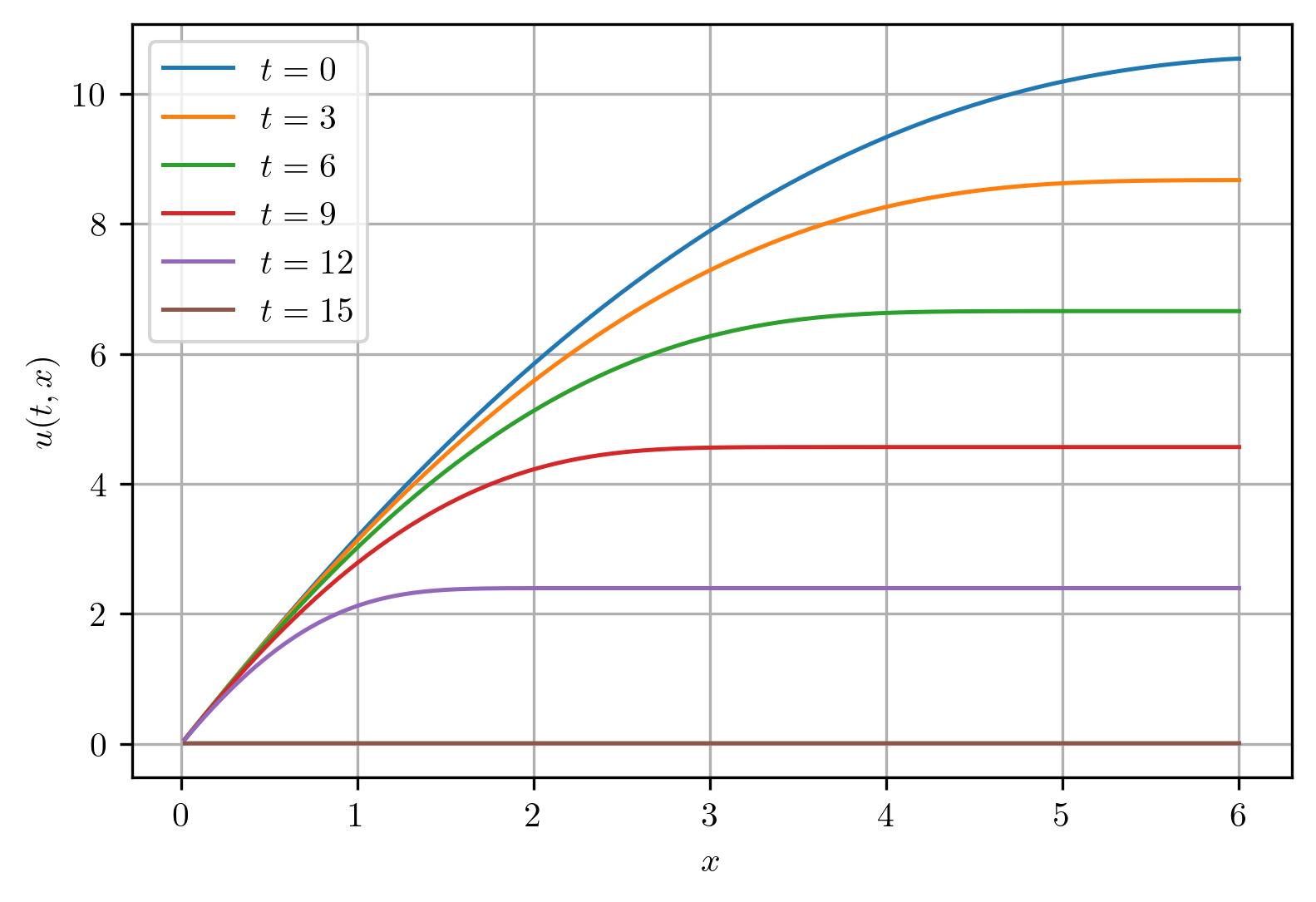} \\
	\end{tabular}
\end{figure}

\begin{figure}[h!]
	\centering
	\caption{Test 1 (BM): optimal control $q^*$ in $2D$ (left) and $3D$ (right)}\label{q 23D}
	\begin{tabular}{cc}
	\includegraphics[width=0.45\textwidth]{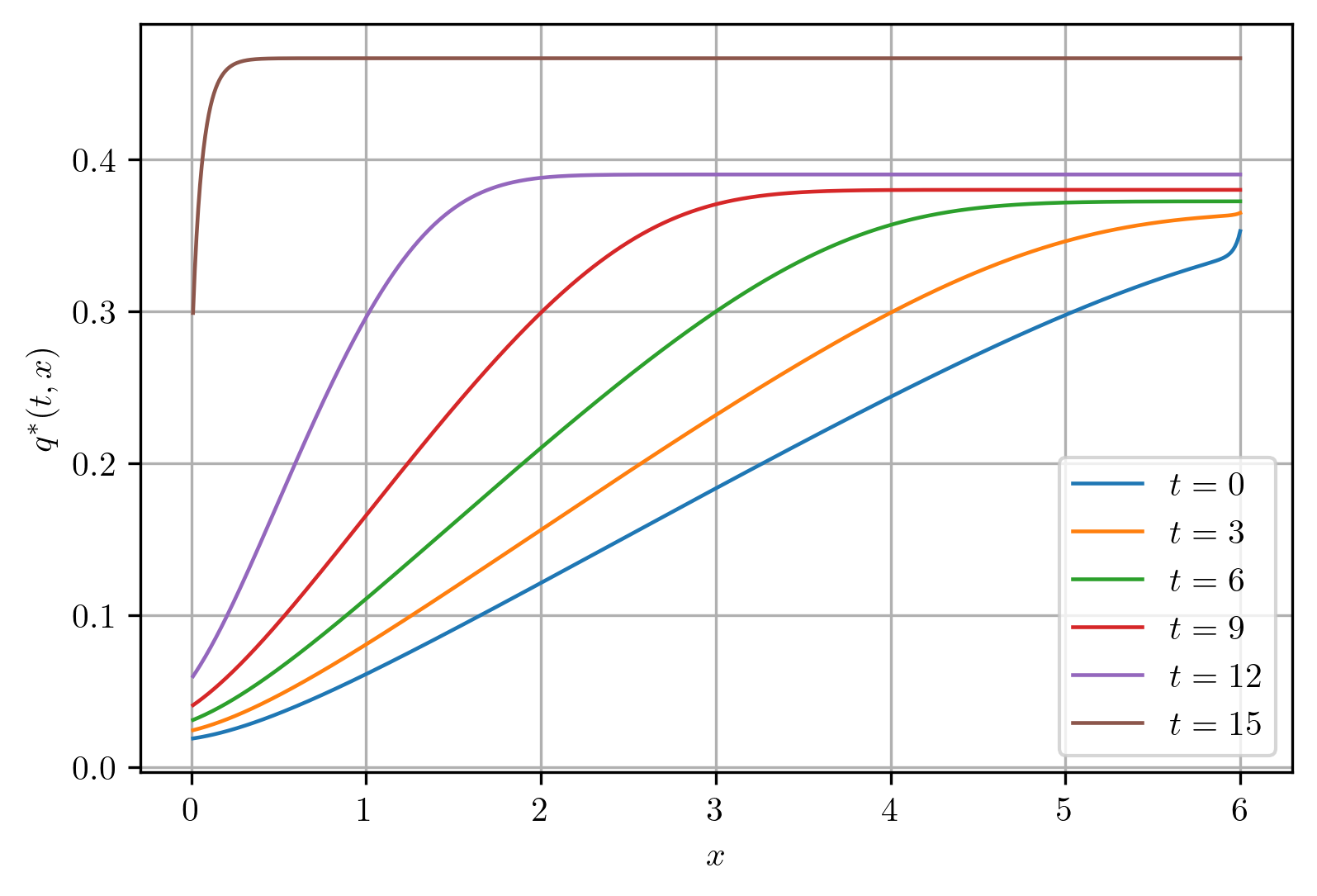} &
	\includegraphics[width=0.45\textwidth]{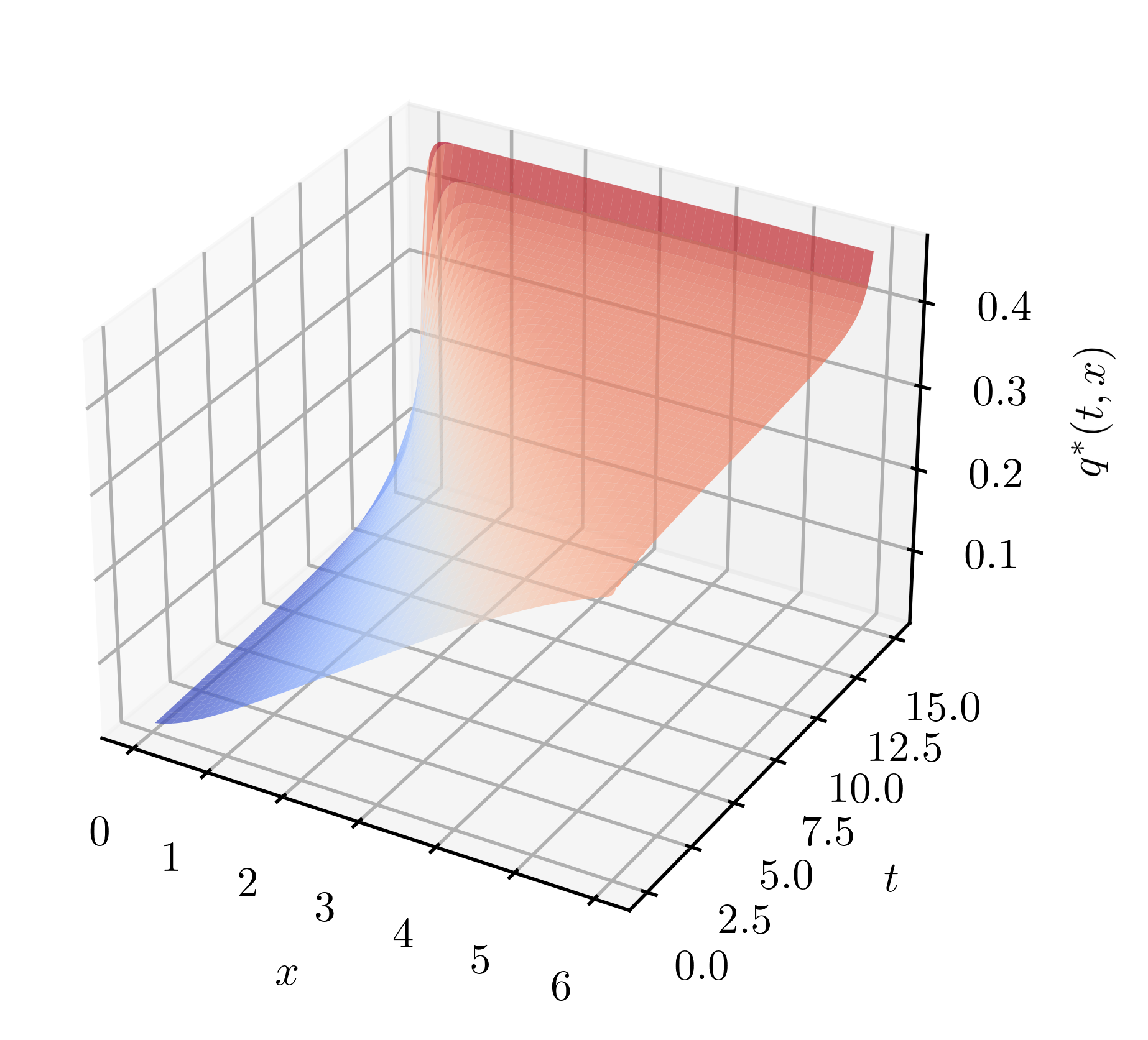} \\
	\end{tabular}
\end{figure}

\begin{figure}[h!]
	\centering
	\caption{Test 1 (GBM): convergence}\label{anG}
	\begin{tabular}{cc}
	\includegraphics[width=0.45\textwidth]{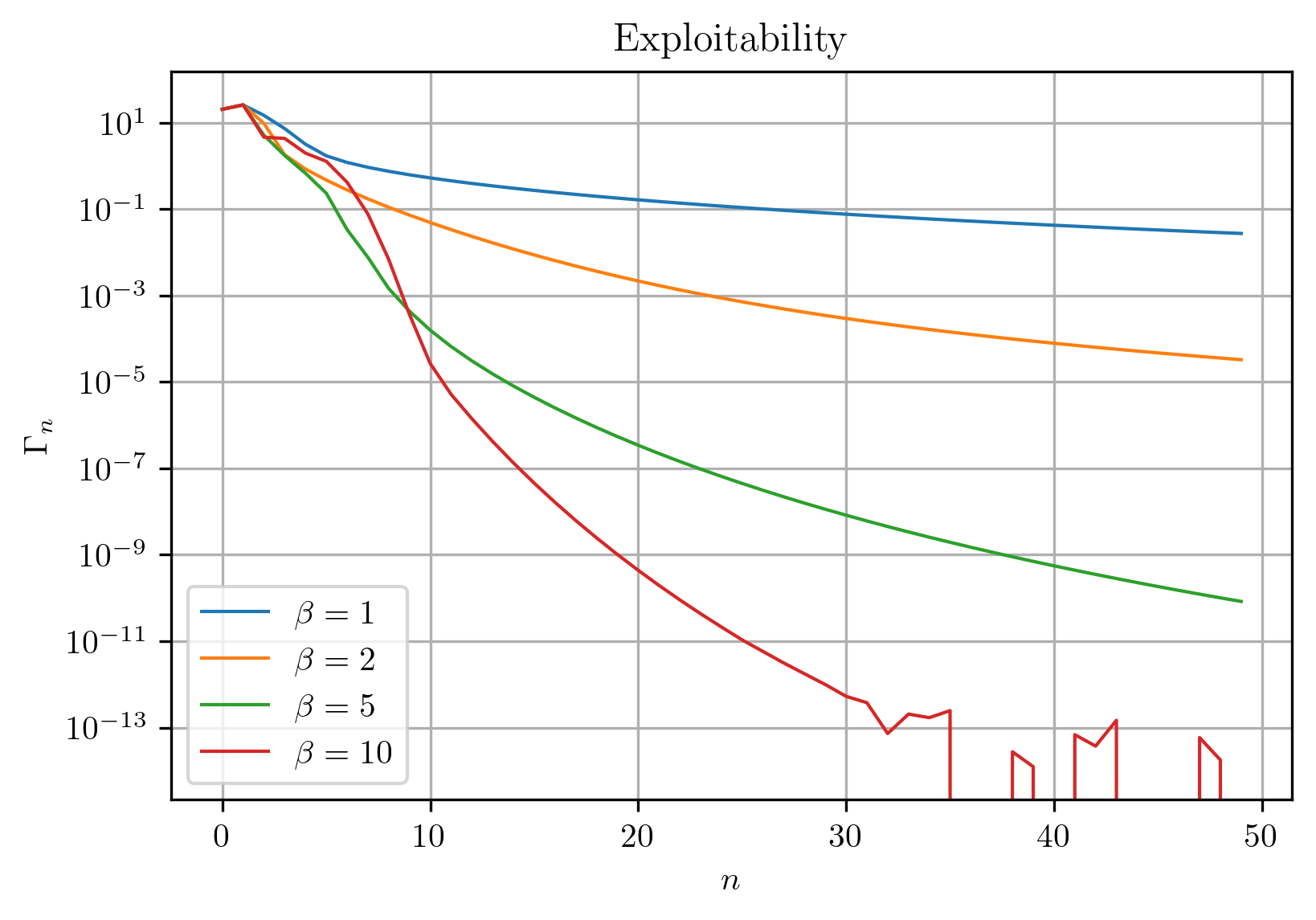} &
	\includegraphics[width=0.45\textwidth]{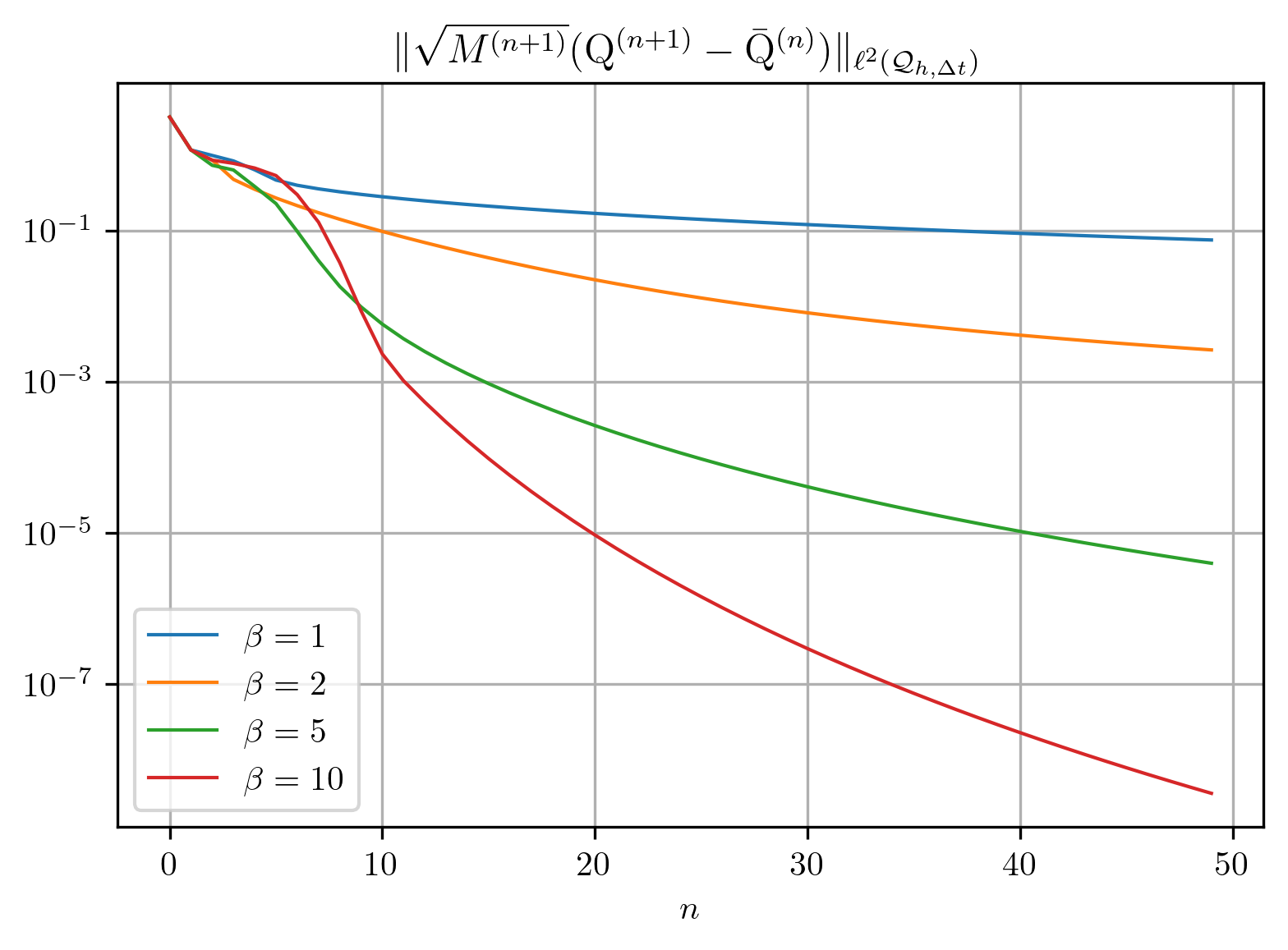} \\
	\end{tabular}
\end{figure}

\begin{figure}[h!]
	\centering
	\caption{Solution to test 1 (GBM)}\label{test1 solG}
	\begin{tabular}{cc}
	\includegraphics[width=0.45\textwidth]{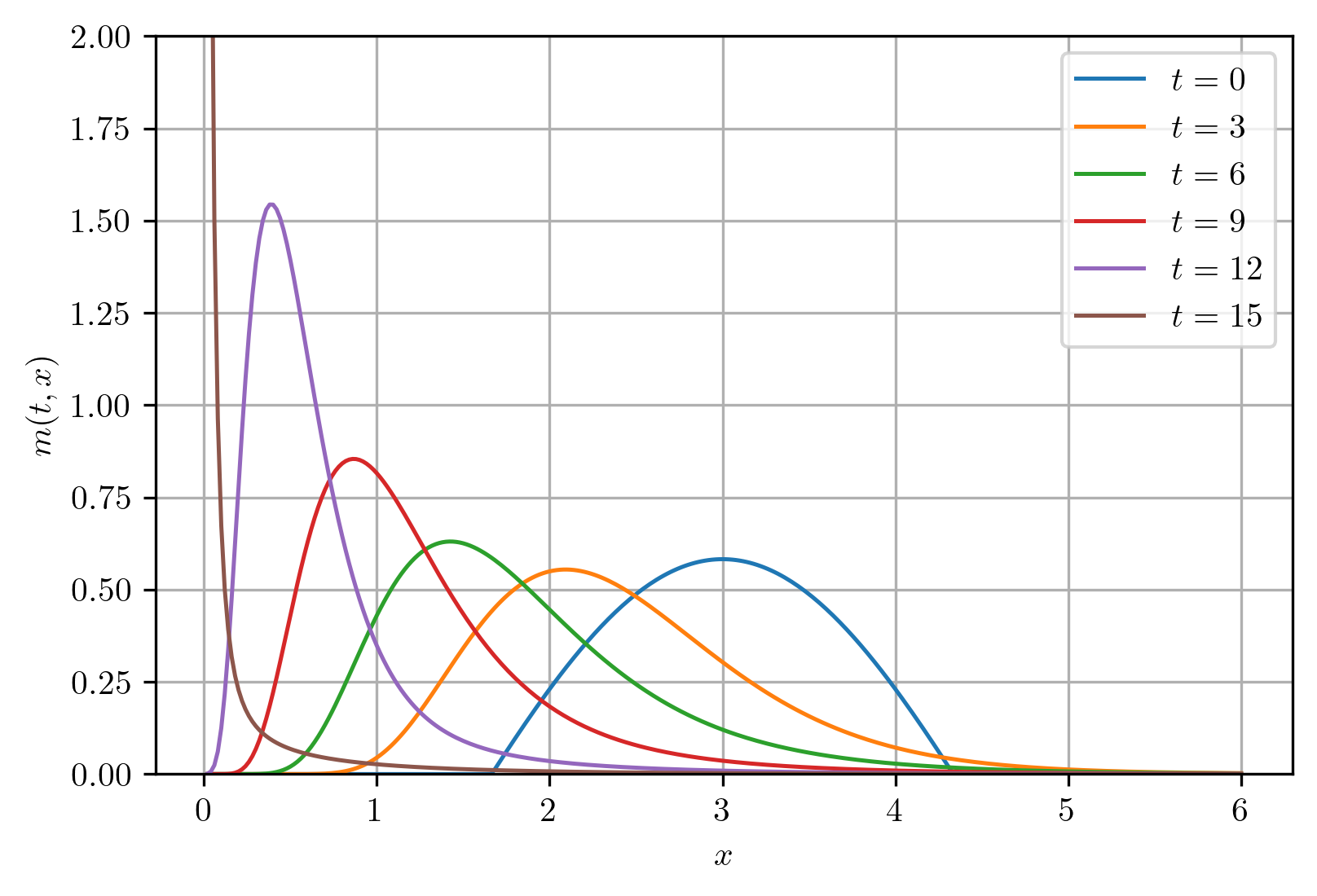} &
	\includegraphics[width=0.45\textwidth]{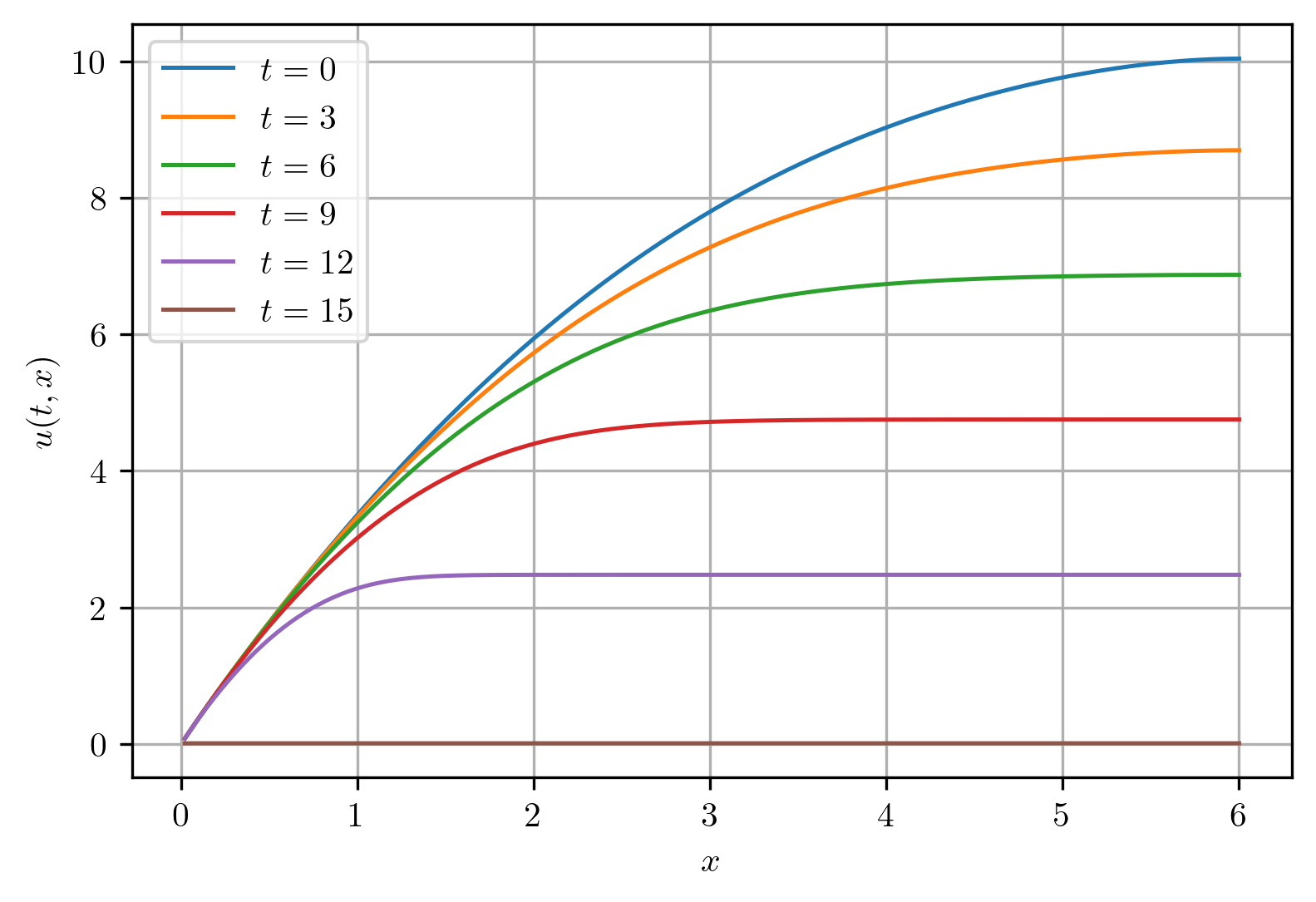} \\
	\end{tabular}
\end{figure}

\begin{figure}[h!]
	\centering
	\caption{Test 1 (GBM): optimal control $q^*$ in $2D$ (left) and $3D$ (right)}\label{testG q 23D}
	\begin{tabular}{cc}
	\includegraphics[width=0.45\textwidth]{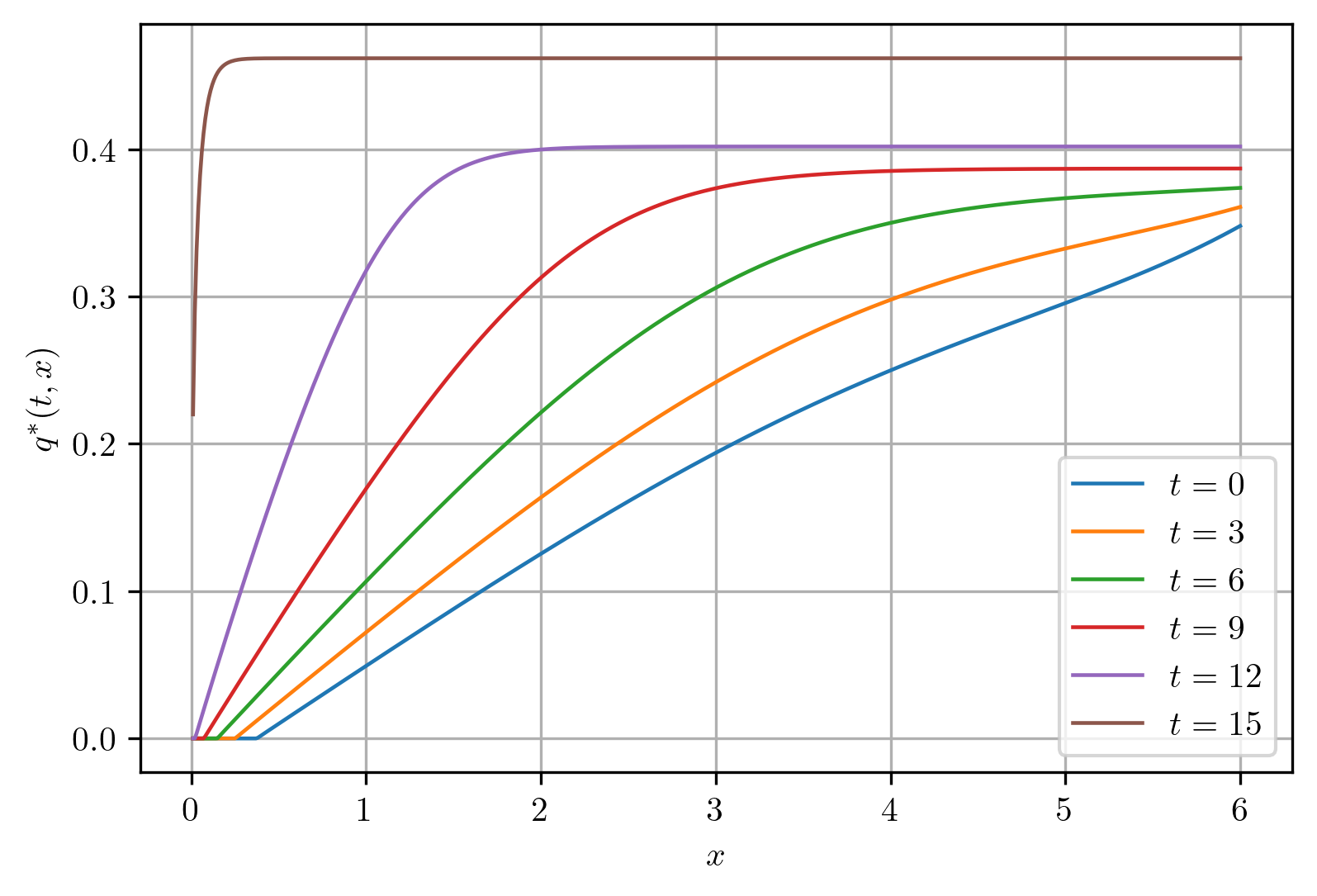} &
	\includegraphics[width=0.45\textwidth]{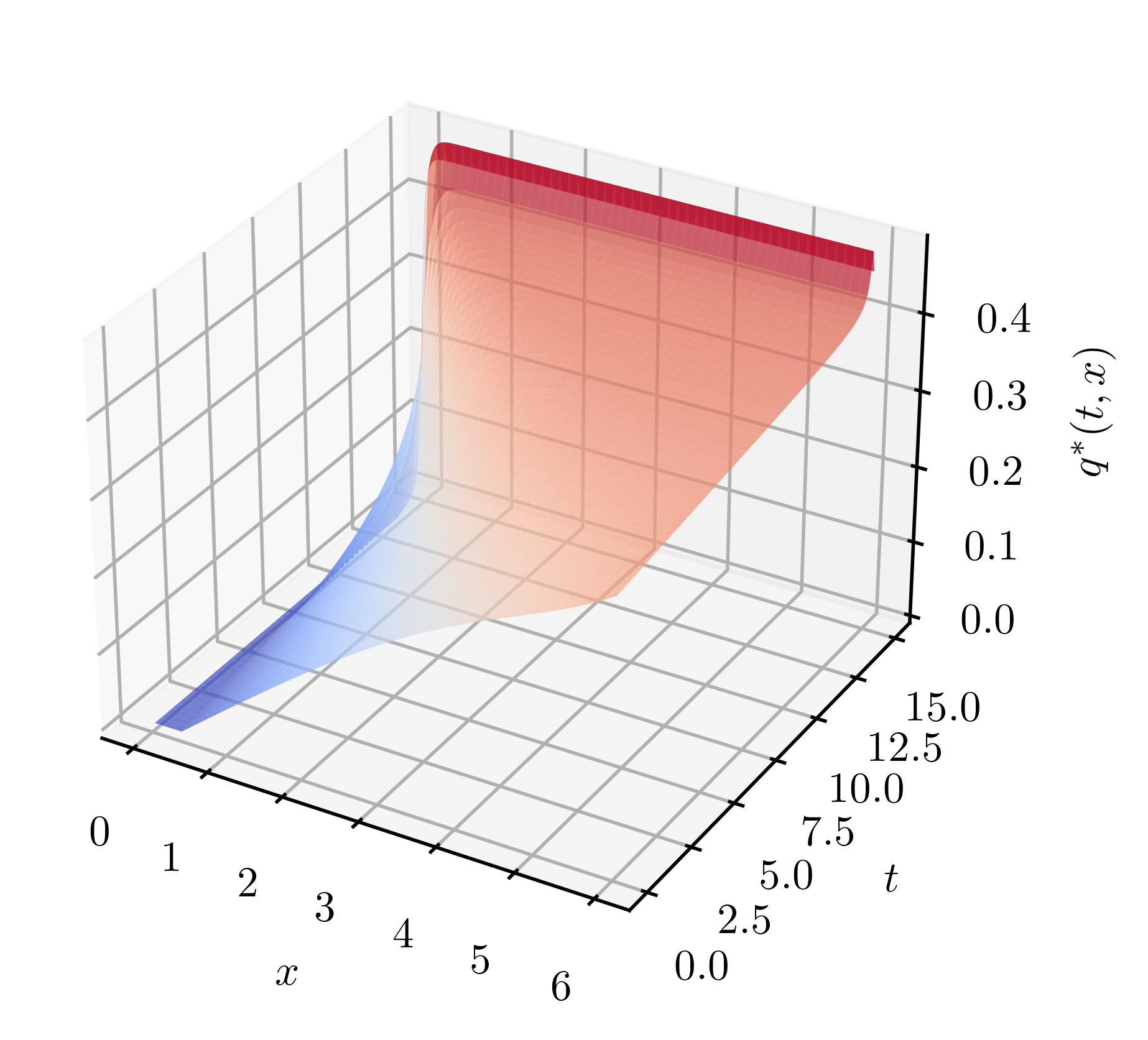} \\
	\end{tabular}
\end{figure}

{In \cref{an} (left) and \cref{anG} (left) we plot the decay of exploitability with different learning rates $\zeta_n$ to illustrate empirical convergence. The weighted sum $\|\sqrt{M^{(n+1)}}(\mathrm{Q}^{(n+1)}-\bar{\mathrm{Q}}^{(n)})\|_{\ell^2(\mathcal{Q}_{h,\Delta t})}$ (right plots) is also a useful measure of convergence. It is the discretized form of $\sqrt{a_n}$ defined in \cref{a_n 2}. In \cref{test1 sol1} and \cref{test1 solG}, the density shifts to the left as producers exhaust their resources.} \par
From the definition of the value function in \cref{valuefunction}, it is clear that for $\lambda=0$, we have $u(t_1,x)\geq u(t_2,x)$ if $0\leq t_1<t_2\leq T$. We verify this in \cref{test1 sol1} and \cref{test1 solG}. From \cref{test1 sol1} and \cref{test1 solG} we also observe that $u(t,\cdot)$ a concave function of $x$ for all $t\in [0,T]$. In \cref{q 23D} and \cref{testG q 23D}, we can see $\partial_xq^*(t,x)\geq 0$ which is consistent with $u$ being concave. \par

The GBM case is not completely compatible with \ref{assumption4} for our results, as $\sigma^2(x)$ is degenerate at $x=0$. We can still get reasonable approximations with the numerical method in this paper, as the discretization method is robust even for approximating viscosity solution to degenerate HJB equations. \par
In the next example we consider the oil production model in  \cite[Section 3.2.2]{MR2762362}, see also~\cite{gueant2010mean} for more details. The price function is defined by \cref{DP2} and the diffusion is GBM, i.e. $\sigma^2(x)= (\sigma x)^2$, $\sigma=0.05$. The other parameters are taken as, following \cite[p. 36]{gueant2010mean}:
$
L=60,T=150,\lambda=0.05,\gamma=10,\kappa=50,E=40,\rho=0.02,\eta=1.2,\delta=0.1.
$
For space and time discretization, we use $N_x=600$ and $N_t=1500$ so that $\Delta t = h = 0.1$. The initial condition is taken as 
$$
\displaystyle m_0(x)=\frac{\lc e^{-0.0008(x-30)^2}-0.7\rc_+}{\int_0^L\lc e^{-0.0008(x-30)^2}-0.7\rc_+ dx}.
$$
We have obtained in \cref{price and production} the same results as Gu\'eant, Lasry and Lions~\cite[Fig. 8 and Fig 9, p. 36-37]{gueant2010mean}. \cref{Oilan} shows the convergence of exploitability and $\sqrt{a_n}$. \cref{oil density} shows the evolution of reserve density and total mass over time.  As mentioned in ~\cite{gueant2010mean}, oil production increases and then decreases, which can be viewed as a form of the so-called Hubbert peak. 

\begin{figure}[h!]
	\centering
	\caption{Oil production model: convergence}\label{Oilan}
	\begin{tabular}{cc}
	\includegraphics[width=0.45\textwidth]{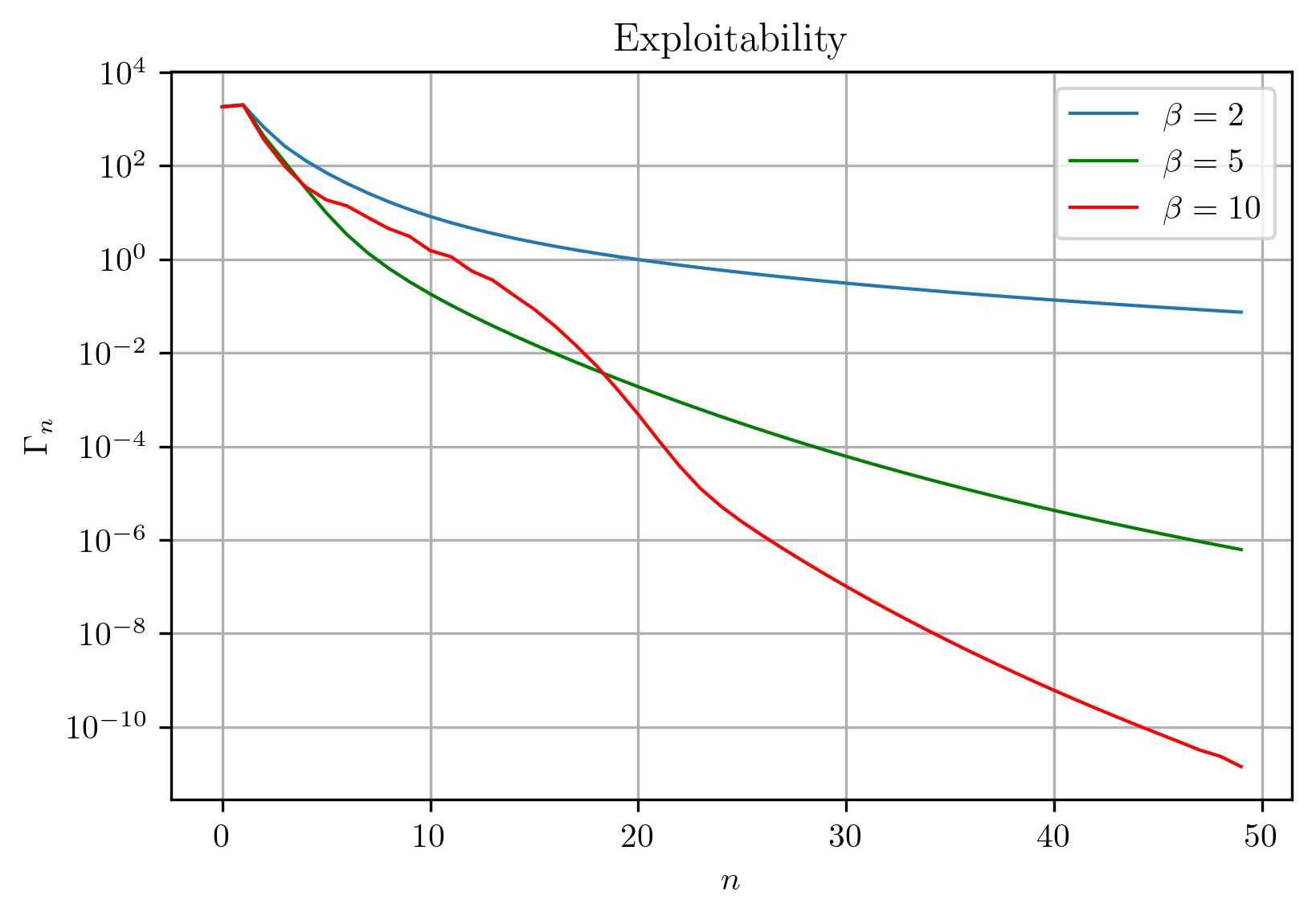} &
	\includegraphics[width=0.45\textwidth]{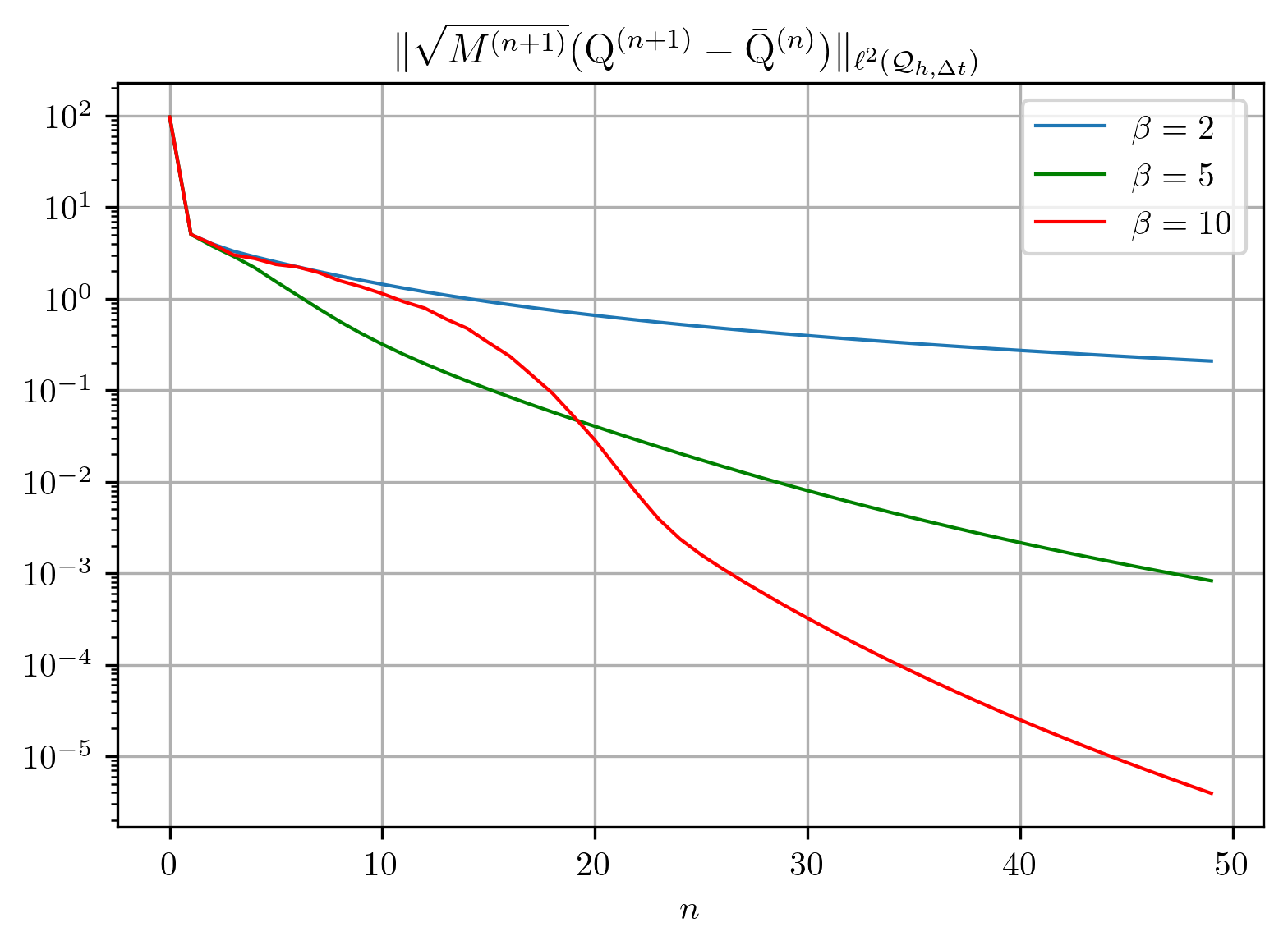} \\
	\end{tabular}
\end{figure}

\begin{figure}[h!]
	\centering
	\caption{ Evolution of price (left) and aggregate production (right)}\label{price and production}
	\begin{tabular}{cc}
	\includegraphics[width=0.45\textwidth]{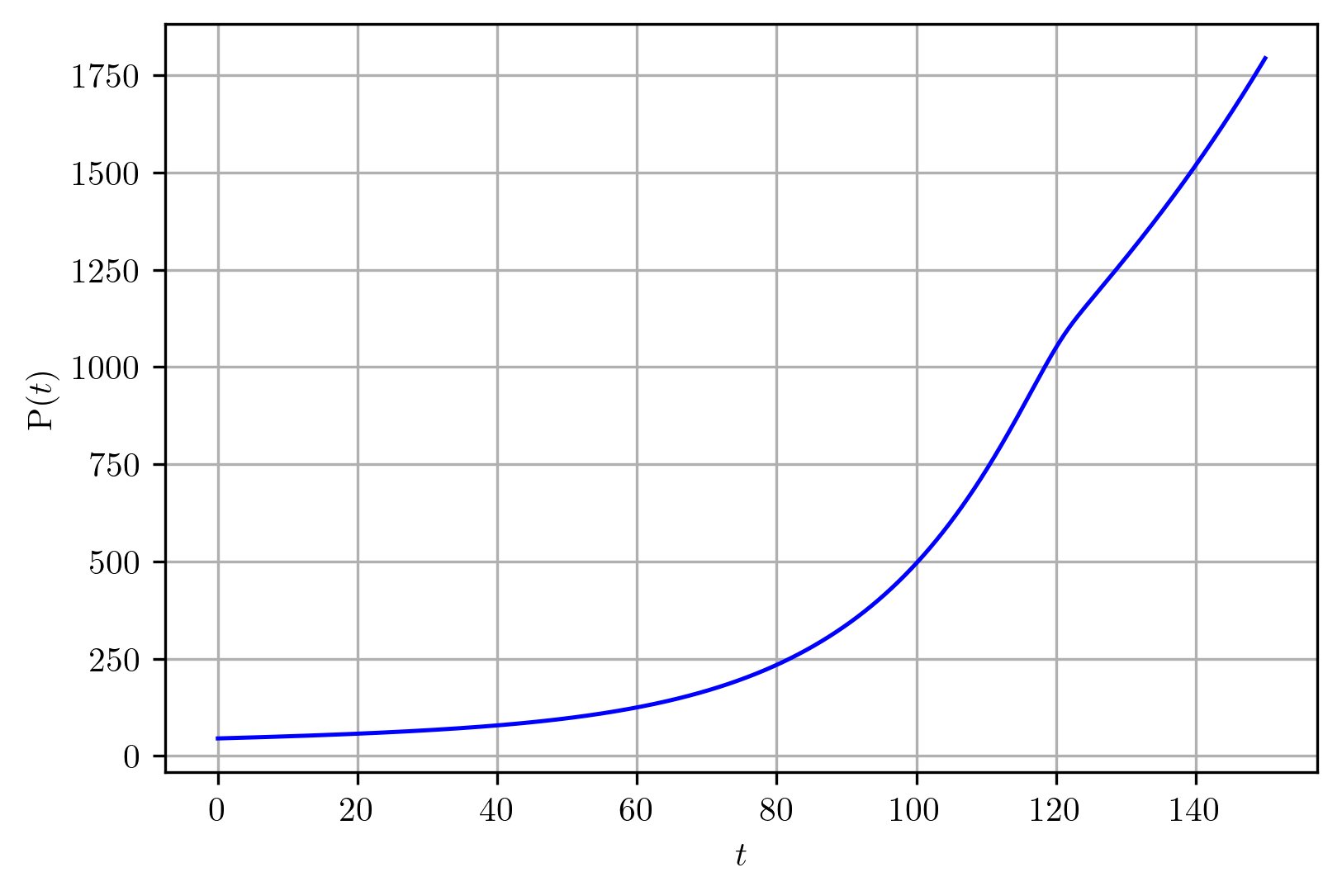} &
	\includegraphics[width=0.45\textwidth]{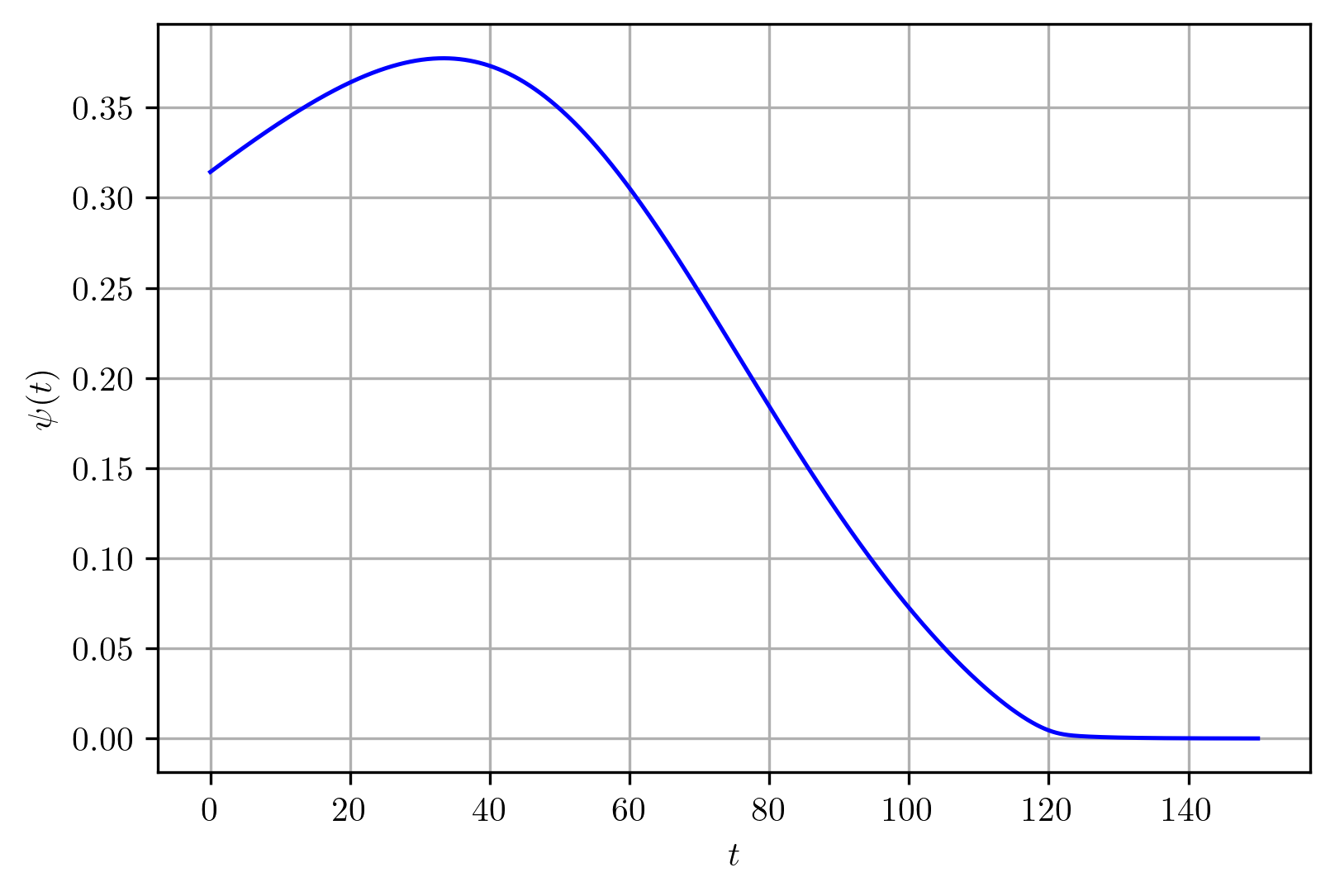} \\
	\end{tabular}
\end{figure}

\begin{figure}[h!tbp]
	\centering
	\caption{Evolution of density in the oil production model}\label{oil density}
	\begin{tabular}{cc}
       \includegraphics[width=0.45\textwidth]{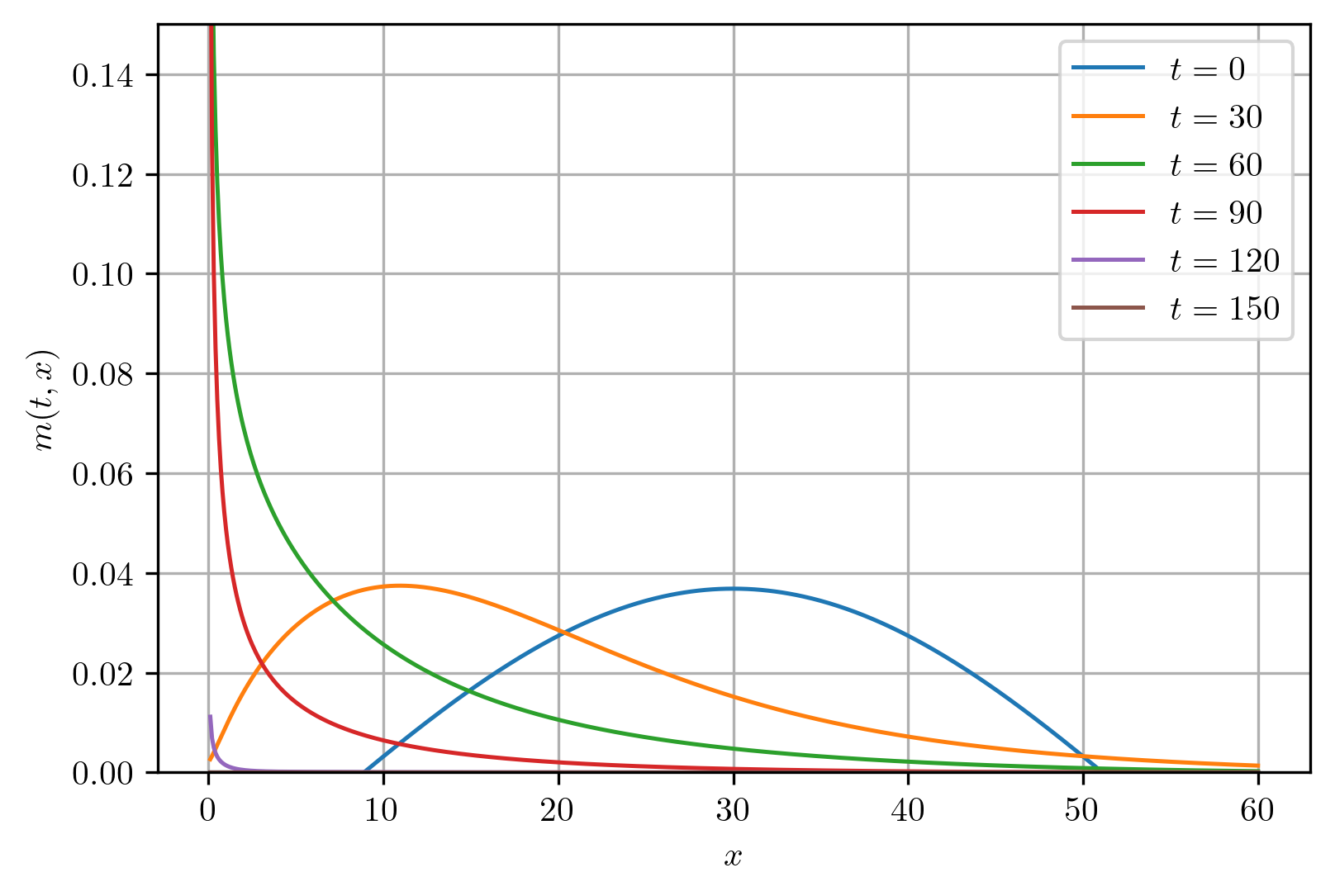} &
	\includegraphics[width=0.45\textwidth]{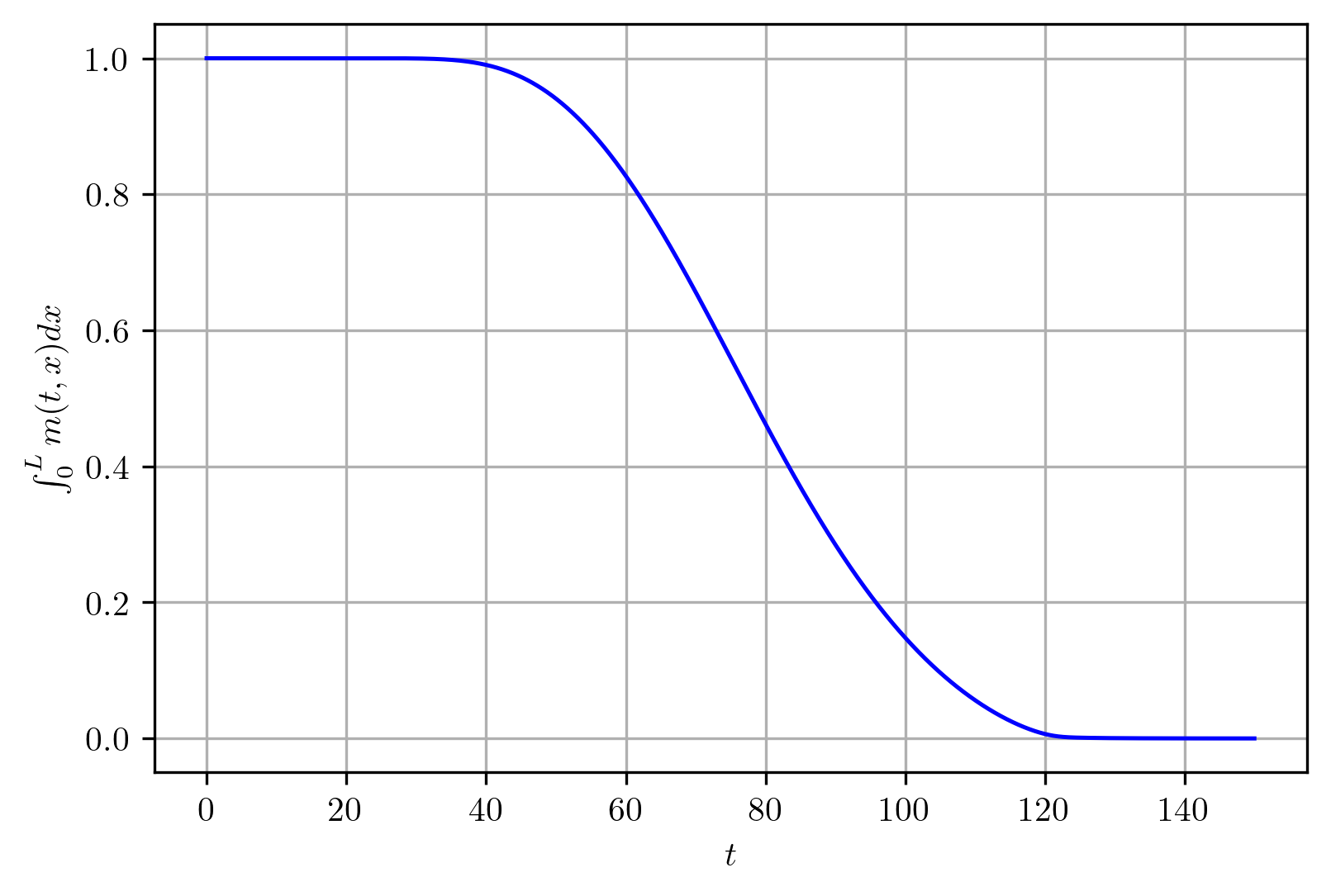} \\
	\end{tabular}
\end{figure}

\pagebreak
\noindent {\bf Acknowledgment.}
{The authors thank Jiahao Song for contributions to the numerical simulations section. We also thank two anonymous referees whose comments helped to improve the previous version of this work. }

\bibliographystyle{siam} 
\bibliography{cournotmfgrefs}

\begin{thebibliography}{10}

\bibitem{MR2888257}
{\sc Y.~Achdou, F.~Camilli, and I.~Capuzzo-Dolcetta}, {\em Mean field games:
  numerical methods for the planning problem}, SIAM J. Control Optim., 50
  (2012), pp.~77--109.

\bibitem{MR2679575}
{\sc Y.~Achdou and I.~Capuzzo-Dolcetta}, {\em Mean field games: numerical
  methods}, SIAM J. Numer. Anal., 48 (2010), pp.~1136--1162.

\bibitem{MR4146720}
{\sc Y.~Achdou and Z.~Kobeissi}, {\em Mean field games of controls: finite
  difference approximations}, Math. Eng., 3 (2021), pp.~Paper No. 024, 35.

\bibitem{MR4214777}
{\sc Y.~Achdou and M.~Lauri\`ere}, {\em Mean field games and applications:
  numerical aspects}, in Mean field games, vol.~2281 of Lecture Notes in Math.,
  Springer, Cham, [2020] \copyright 2020, pp.~249--307.

\bibitem{MR4223355}
{\sc Y.~Achdou, M.~Lauri\`ere, and P.-L. Lions}, {\em Optimal control of
  conditioned processes with feedback controls}, J. Math. Pures Appl. (9), 148
  (2021), pp.~308--341.

\bibitem{MR3888965}
{\sc M.~Bardi and M.~Cirant}, {\em Uniqueness of solutions in mean field games
  with several populations and {N}eumann conditions}, in P{DE} models for
  multi-agent phenomena, vol.~28 of Springer INdAM Ser., Springer, Cham, 2018,
  pp.~1--20.

\bibitem{bonnans2021schauder}
{\sc J.~F. Bonnans, S.~Hadikhanloo, and L.~Pfeiffer}, {\em Schauder estimates
  for a class of potential mean field games of controls}, Appl. Math. Optim.,
  83 (2021), pp.~1431--1464.

\bibitem{brezis2011functional}
{\sc H.~Brezis}, {\em {Functional Analysis, Sobolev Spaces and Partial
  Differential Equations}}, 2011.

\bibitem{MR4291367}
{\sc S.~Cacace, F.~Camilli, and A.~Goffi}, {\em A policy iteration method for
  mean field games}, ESAIM Control Optim. Calc. Var., 27 (2021), pp.~Paper No.
  85, 19.

\bibitem{MR3601001}
{\sc S.~Cacace, F.~Camilli, and C.~Marchi}, {\em A numerical method for mean
  field games on networks}, ESAIM Math. Model. Numer. Anal., 51 (2017),
  pp.~63--88.

\bibitem{MR4565014}
{\sc F.~Camilli and C.~Marchi}, {\em On quasi-stationary mean field games of
  controls}, Appl. Math. Optim., 87 (2023), pp.~Paper No. 47, 31.

\bibitem{MR4392286}
{\sc F.~Camilli and Q.~Tang}, {\em Rates of convergence for the policy
  iteration method for mean field games systems}, J. Math. Anal. Appl., 512
  (2022), pp.~Paper No. 126138, 18.

\bibitem{MR3967062}
{\sc P.~Cardaliaguet, F.~Delarue, J.-M. Lasry, and P.-L. Lions}, {\em The
  master equation and the convergence problem in mean field games}, vol.~201 of
  Annals of Mathematics Studies, Princeton University Press, Princeton, NJ,
  2019.

\bibitem{MR3608094}
{\sc P.~Cardaliaguet and S.~Hadikhanloo}, {\em Learning in mean field games:
  the fictitious play}, ESAIM Control Optim. Calc. Var., 23 (2017),
  pp.~569--591.

\bibitem{MR3805247}
{\sc P.~Cardaliaguet and C.-A. Lehalle}, {\em Mean field game of controls and
  an application to trade crowding}, Math. Financ. Econ., 12 (2018),
  pp.~335--363.

\bibitem{MR3752669}
{\sc R.~Carmona and F.~Delarue}, {\em Probabilistic theory of mean field games
  with applications. {I}}, vol.~83 of Probability Theory and Stochastic
  Modelling, Springer, Cham, 2018.
\newblock Mean field FBSDEs, control, and games.

\bibitem{MR3359708}
{\sc P.~Chan and R.~Sircar}, {\em Bertrand and {C}ournot mean field games},
  Appl. Math. Optim., 71 (2015), pp.~533--569.

\bibitem{chan2017fracking}
\leavevmode\vrule height 2pt depth -1.6pt width 23pt, {\em Fracking,
  renewables, and mean field games}, SIAM Review, 59 (2017), pp.~588--615.

\bibitem{MR3333058}
{\sc M.~Cirant}, {\em Multi-population mean field games systems with {N}eumann
  boundary conditions}, J. Math. Pures Appl. (9), 103 (2015), pp.~1294--1315.

\bibitem{deschamps1975algorithm}
{\sc R.~Deschamps}, {\em An algorithm of game theory applied to the duopoly
  problem}, European Economic Review, 6 (1975), pp.~187--194.

\bibitem{MR3160525}
{\sc D.~A. Gomes, S.~Patrizi, and V.~Voskanyan}, {\em On the existence of
  classical solutions for stationary extended mean field games}, Nonlinear
  Anal., 99 (2014), pp.~49--79.

\bibitem{MR3559742}
{\sc D.~A. Gomes, E.~A. Pimentel, and V.~Voskanyan}, {\em Regularity theory for
  mean-field game systems}, vol.~1 of SpringerBriefs in Mathematics, Springer,
  [Cham], 2016.

\bibitem{MR3755719}
{\sc P.~J. Graber and A.~Bensoussan}, {\em Existence and uniqueness of
  solutions for {B}ertrand and {C}ournot mean field games}, Appl. Math. Optim.,
  77 (2018), pp.~47--71.

\bibitem{MR4223351}
{\sc P.~J. Graber, V.~Ignazio, and A.~Neufeld}, {\em Nonlocal {B}ertrand and
  {C}ournot mean field games with general nonlinear demand schedule}, J. Math.
  Pures Appl. (9), 148 (2021), pp.~150--198.

\bibitem{MR3888969}
{\sc P.~J. Graber and C.~Mouzouni}, {\em Variational mean field games for
  market competition}, in P{DE} models for multi-agent phenomena, vol.~28 of
  Springer INdAM Ser., Springer, Cham, 2018, pp.~93--114.

\bibitem{MR4064472}
\leavevmode\vrule height 2pt depth -1.6pt width 23pt, {\em On mean field games
  models for exhaustible commodities trade}, ESAIM Control Optim. Calc. Var.,
  26 (2020), pp.~Paper No. 11, 38.

\bibitem{MR4506079}
{\sc P.~J. Graber and R.~Sircar}, {\em Master equation for {C}ournot mean field
  games of control with absorption}, J. Differential Equations, 343 (2023),
  pp.~816--909.

\bibitem{gueant2010mean}
{\sc O.~Gu{\'e}ant, J.-M. Lasry, and P.~L. Lions}, {\em {Mean Field Games and
  Oil Production}}, in {The Economics of Sustainable Development}, Economica,
  ed., 2010.

\bibitem{MR2762362}
{\sc O.~Gu\'{e}ant, J.-M. Lasry, and P.-L. Lions}, {\em Mean field games and
  applications}, in Paris-{P}rinceton {L}ectures on {M}athematical {F}inance
  2010, vol.~2003 of Lecture Notes in Math., Springer, Berlin, 2011,
  pp.~205--266.

\bibitem{GYONGY2024857}
{\sc I.~Gy{\"o}ngy and S.~Kim}, {\em Harnack inequality for parabolic equations
  in double-divergence form with singular lower order coefficients}, J.
  Differential Equations, 412 (2024), pp.~857--880.

\bibitem{MR4030259}
{\sc S.~Hadikhanloo and F.~J. Silva}, {\em Finite mean field games: fictitious
  play and convergence to a first order continuous mean field game}, J. Math.
  Pures Appl. (9), 132 (2019), pp.~369--397.

\bibitem{MR2352434}
{\sc M.~Huang, P.~E. Caines, and R.~P. Malham\'{e}}, {\em Large-population
  cost-coupled {LQG} problems with nonuniform agents: individual-mass behavior
  and decentralized {$\epsilon$}-{N}ash equilibria}, IEEE Trans. Automat.
  Control, 52 (2007), pp.~1560--1571.

\bibitem{kobeissi2022mean}
{\sc Z.~Kobeissi}, {\em Mean field games with monotonous interactions through
  the law of states and controls of the agents}, Nonlinear Differ. Equ. Appl.
  NoDEA, 29 (2022), p.~52.

\bibitem{MR4387201}
\leavevmode\vrule height 2pt depth -1.6pt width 23pt, {\em On classical
  solutions to the mean field game system of controls}, Comm. Partial
  Differential Equations, 47 (2022), pp.~453--488.

\bibitem{MR0241822}
{\sc O.~A. Lady\v{z}enskaia, V.~A. Solonnikov, and N.~N. Ural\'ceva}, {\em
  Linear and quasilinear equations of parabolic type}, vol.~Vol. 23 of
  Translations of Mathematical Monographs, American Mathematical Society,
  Providence, RI, 1968.
\newblock Translated from the Russian by S. Smith.

\bibitem{MR2295621}
{\sc J.-M. Lasry and P.-L. Lions}, {\em Mean field games}, Jpn. J. Math., 2
  (2007), pp.~229--260.

\bibitem{MR4368188}
{\sc M.~Lauri\`ere}, {\em Numerical methods for mean field games and mean field
  type control}, in Mean field games, vol.~78 of Proc. Sympos. Appl. Math.,
  Amer. Math. Soc., Providence, RI, [2021] \copyright 2021, pp.~221--282.

\bibitem{MR4534442}
{\sc M.~Lauri\`ere, J.~Song, and Q.~Tang}, {\em Policy iteration method for
  time-dependent mean field games systems with non-separable {H}amiltonians},
  Appl. Math. Optim., 87 (2023), pp.~Paper No. 17, 34.

\bibitem{MR4659381}
{\sc P.~Lavigne and L.~Pfeiffer}, {\em Generalized conditional gradient and
  learning in potential mean field games}, Appl. Math. Optim., 88 (2023),
  pp.~Paper No. 89, 36.

\bibitem{perrin2020fictitious}
{\sc S.~Perrin, J.~P{\'e}rolat, M.~Lauri{\`e}re, M.~Geist, R.~Elie, and
  O.~Pietquin}, {\em Fictitious play for mean field games: Continuous time
  analysis and applications}, Advances in neural information processing
  systems, 33 (2020), pp.~13199--13213.

\bibitem{tang2023learning}
{\sc Q.~Tang and J.~Song}, {\em Learning optimal policies in potential mean
  field games: Smoothed policy iteration algorithms}, SIAM J. of Control \&
  Optimization, 62 (2024), pp.~351--375.

\bibitem{MR1039721}
{\sc L.~Thorlund-Petersen}, {\em Iterative computation of {C}ournot
  equilibrium}, Games Econom. Behav., 2 (1990), pp.~61--75.

\end{thebibliography}

\end{document}